\documentclass[11pt]{amsart}
\usepackage{amsfonts}
\usepackage{latexsym}
\usepackage{amscd}
\usepackage{amsmath}
\usepackage{amssymb}
\usepackage{graphicx}
\usepackage[mathscr]{eucal}
\theoremstyle{plain}
\newtheorem{thm}{Theorem}[section]
\newtheorem{lem}{Lemma}[section]
\newtheorem{prop}{Proposition}[section]
\newtheorem{cor}{Corollary}[section]

\theoremstyle{definition}
\newtheorem{rem}{Remark}[section]
\newtheorem{defn}{Definition}[section]
\newtheorem{ex}{Example}[section]
\newtheorem{exs}{Examples}[section]
\newcommand{\A}{\mathcal A}
\newcommand{\B}{\mathcal B}
\newcommand{\C}{\mathcal C}

\newcommand{\E}{\mathcal E}
\newcommand{\I}{\mathcal I}
\newcommand{\J}{\mathcal J}
\newcommand{\LL}{\mathcal L}
\newcommand{\K}{\mathcal K}

\newcommand{\N}{\mathcal N}
\newcommand{\OO}{\mathcal O}
\newcommand{\PP}{\mathcal P}

\newcommand{\V}{\mathcal V}

\newcommand{\hs}{\hspace{3mm}}
\newcommand{\vs}{\vspace{1mm}}
\newcommand{\vsd}{\vspace{2mm}}
\newcommand{\st}{\hspace{2mm}|\hspace{2mm}}
\newcommand{\XC}{X_\C}
\newcommand{\um}{\underline m}
\def\som{\circle*{0.13}}
\newcommand{\bD}{\mathbb D}
\newcommand{\bK}{\mathbb K}

\newcommand{\bP}{\mathbb P}
\newcommand{\bR}{\mathbb R}
\newcommand{\bZ}{\mathbb Z}
\newcommand{\ga}{\alpha}
\newcommand{\De}{\Delta}

\newcommand{\g}{{\gamma}}
\newcommand{\G}{{\Gamma}}
\newcommand{\GC}{\Gamma_{\C}}

\newcommand{\ff}{\varphi}

\newcommand{\s}{\sigma}
\newcommand{\Si}{\Sigma}

\newcommand{\prox}{\rightarrow}
\newcommand{\pprox}{\leadsto}
\newcommand{\lprox}{\twoheadrightarrow}
\newcommand{\llprox}{\stackrel{\ell}{\lprox}}

\newcommand{\co}{{\mathcal O}}

\newcommand{\gp}{\mathbb{P}}
\newcommand{\gz}{\mathbb{Z}}

\newcommand{\gr}{\mathbb{R}}
\newcommand{\gc}{\mathbb{C}}

\newcommand{\cL}{{\mathcal L}}
\newcommand{\cf}{{\mathcal F}}

\renewcommand{\int}{{\rm int}}

\newcommand{\edim}{{\rm edim\;}}
\def\som{\circle*{0.13}}
\begin{document}
\title[]
{Configurations of infinitely near points}
\author[]{A. Campillo, G. Gonzalez-Sprinberg, F. Monserrat }
\thanks{The third author is supported by MEC
 MTM2007-64704 and Bancaixa P1-1A2005-08}
\address{}
\email{campillo@agt.uva.es, gonsprin@ujf-grenoble.fr,
framonde@mat.upv.es}
\begin{abstract}
We present a survey of some aspects and new results on
configurations, i.e. disjoint unions of constellations of
infinitely near points, local and global theory, with some
applications and results on generalized Enriques diagrams,
singular foliations, and
 linear systems defined by clusters.
\end{abstract}
\maketitle
\section{Introduction}
\label{sec:intro} The subject in this paper can be understood to
be originated with the study of singularities of plane curves,
their desingularization by iterated point blowing-ups, and the
problems of existence of curves having either assigned
singularities or passing through a given set of points or
infinitely near points with prescribed multiplicities.\vs

Early, in the past century, Enriques gave an answer for the
existence of such curves, without conditions on the degree, in terms
of precise inequalities involving the prescribed multiplicities.
Some essential data for it, are the proximity relations among the
given infinitely near points, i.e. the incidence between points and
the transforms of exceptional divisors obtained by blowing-up
precedent points. This data is encoded in the so called Enriques
diagrams.\vs

Later, the local study of (complete) linear systems of curves
leads Zariski, in the thirties, to define complete ideals on
regular local rings, investigate their structure and establish
their theory in the smooth two dimensional case. Lipman has
continued the development of the theory of complete ideals for
singular two dimensional cases and in higher dimensions; in
particular, in the eighties, by establishing it for finitely
supported complete ideals, i.e. for ideals supported at the closed
point and such that there exist finite sequences of point
blowing-ups which make the ideals locally principal. In parallel,
also in the later eighties, Casas develops in modern geometrical
terms the two dimensional theory,  and applies it to solve the
problem of determining multiplicities of passage through
infinitely near points for polar curves of plane curve
singularities. \vs

Factorization theorems are obtained, in general, for these ideals
in terms of simple (or special *-simple) ideals corresponding to
finite chains of infinitely near points. An algebraic-geometric
point of view, in terms of geometry of infinitely near points, was
given by the first two authors and Lejeune-Jalabert in the
nineties. This includes the treatment of finitely supported toric
ideals, and a rather explicit theory for them.\vs

Applications of finitely supported complete ideals have been
developed, both in the local and global cases, by means of
constellations and configurations of infinitely near points
respectively. Such applications include those of different subjects
as, among others, the study of singularities of adjoints or polar
curves \cite{Ca}, \cite{casas}, \cite{delgado}, \cite{L4},
\cite{L5}, the Poincar\'e problem on the degree of projective
integral curves of first order algebraic differential equations
\cite{Ca-Ca},\cite{Ca-Ca-GF},\cite{C-G-GF-R},\cite{K-E},\cite{l-n},\cite{galmon2},
the Harbourne-Hirschowitz conjecture on special linear systems of
projective plane curves (and related topics) \cite{roe3},
\cite{roe2}, \cite{roe1}, \cite{Mir}, \cite{mons}, or the monodromy
and related conjectures for non degenerated hypersurfaces with
respect to clusters, \cite{lem1}, \cite{lem}.\vs

This paper gives an introduction to this subject as well as a
survey on the main results and several of its
applications. Some new results are also included.\vs

The paper is organized as follows.\vs

In Section \ref{sec:2} we give a survey on the local aspects of
the theory. After recalling the basic definitions of the geometric
theory extending to higher dimensions the classical case presented
in \cite{EC}, we relate this geometric framework with the
algebraic theory  introduced in \cite{Z,ZS}. It follows the
Lipman's unique factorization result (allowing negative exponents)
for finitely supported complete ideals \cite{L2} on a germ of
smooth variety, which extends to higher dimensions the unique
factorization of any complete ideal of a two dimensional local
ring into simple ideals. Each simple ideal in the factorization is
associated to one infinitely near point to the origin, and it
reflects geometrical properties of the chain of points which need
to be blown up in order to create it. A union of such chains with
the same origin will be called constellation along the paper.
Constellations with integral weights at their points are called
clusters. Complete ideals supported on a constellation correspond
to concrete clusters called idealistic. Such clusters with the
natural semigroup structure are called the Galaxy of the
constellation. Galaxies become closed under taking adjoints, and
the part consisting of idealistic clusters
corresponding to adjoint ideals is determined from the Galaxy itself.
\vs

Assuming that the characteristic of the ground field is zero, the
morphism $\s_{\C}: X_{\C} \rightarrow X $ obtained by composition
of the blowing-up of the points of a constellation ${\mathcal C}$ on
the smooth germ $X$ is shown to be an embedded resolution of
complete intersections defined by general elements in a finitely
supported complete ideal with that constellation as support. The
variety $X_{\C}$ will be called the sky of ${\mathcal C}$. On the
other hand, in Section \ref{sect:24} it is shown how the use of
characteristic cones provides a natural framework to study
factorizations properties for ideals.\vs

Finally, we consider the case of toric varieties and monomial
ideals, giving explicit description and results on toric
constellations and proximity in combinatorial terms, which are not
available in general non toric cases. Linear proximity, a  finer
concept than the one of proximity, is also characterized and used
for describing the Galaxy, the characteristic cone and
factorization properties. For the ideals in the subgalaxy
generated by the simple factors in Lipman factorization the whole
features of Zariski's theory hold as in the smooth two dimensional
case. Generalized Enriques diagrams, i.e. those which also
contain the linear proximity information, are also characterized.

Section \ref{sec:3} is devoted, on the one hand, to extend the
language and theory of constellations of infinitely near points
and clusters to the global situation. Here, we consider
configurations, i.e. finite union of constellations with origin at
different points of a smooth variety. The notions of sky, clusters
and Galaxy have an obvious sense also for configurations. On the
other hand, we focus our attention on the study of the cone of
curves $NE(Z)$ of a projective regular surface $Z$, mainly bearing
in mind the case in which $Z$ is rational. We give a description
of some generalities and known properties of $NE(Z)$, showing a
variety of different shapes that it may have, and, in particular,
we focus on the case of finite generation (polyhedrality).
Rational surfaces with polyhedral cone of curves are interesting
issues that have several applications, as we shall see in the last
section. \vs

The so called P-sufficient configurations, introduced in
\cite{galmon} and \cite{galmon3}, are configurations over a
relatively minimal rational surface $X$ satisfying a numerical
condition which depends only on their P-Enriques diagram. The
interest of such a configuration ${\mathcal C}$ is given by the
fact that its sky $X_{\mathcal C}$ has a polyhedral cone of
curves. We consider also configurations ${\mathcal C}$ of base
points of 1-dimensional linear systems (pencils) $\varrho$ on a
projective regular surface. We give a description of the face of
$NE(X_{\mathcal C})$ generated by the classes of the integral
components of the curves in $\varrho$, using it to give (in
characteristic zero) a characterization of the irreducible pencils
(that is, those with integral general curves) in terms of their
clusters of base points. When $\varrho$ is a pencil at infinity
a great deal of information is known on the cone of curves and the
characteristic cone of $X_{\mathcal C}$ \cite{CPR, CPR2} and we
summarize it. In Section \ref{aplications} we show some of
the above mentioned applications of the theory of clusters of
infinitely near points in the global case. \vs

In Section \ref{foliations} we describe some results on the theory
of foliations based on aspects in preceding sections. All these
results are related to the classical Poincar\'e problem on
establishing bounds for the degree of projective curves which are
invariant by an algebraic plane foliation. In Section \ref{hh} we
show that the language of infinitely near points and idealistic
clusters can also be applied to give new results on a conjecture
(the Harbourne-Hirschowitz Conjecture) which deals with the
dimension of the linear systems of the projective plane defined by
clusters whose associated configuration is a set of general points
of the plane. In both cases, there is a very extensive literature
giving either partial proofs for problem or for the conjecture or
dealing with related subjects. We recall some results in cases for
which the Poincar\'e problem or the Harbourne-Hirschowitz Conjecture
have a satisfactory answer. Those results are established in terms
of clusters and their proofs involve the knowledge of the cone of
curves of the surfaces associated with certain pencils at infinity
and certain properties of the P-sufficient configurations.
\section{Local theory}\label{sec:2}
\subsection{Constellations, proximity and Enriques diagrams.}
  Let $X$ be a regular variety of dimension $d$ at least two, over an algebraically
closed field $\bK$. In the sequel we consider varieties obtained from $X$ by a finite
sequence of closed point blowing-ups. A point $P$ is {\em infinitely near} $Q \in X$ if
$Q$ is the image of $P$ under the composition of blowing-ups; denote this relation by $P \geq Q$.
A geometric description of the blowing-up of a point
may be given as an avatar of the graph construction related to the definition of
projective space.
Let $\ff : ({\bK}^d \setminus \{O\}) \prox {\bP}^{d-1}$ with $\ff(x)$
the line joining $O$ to $x$.
Consider the closure $\overline{\G}_{\ff}$ of the graph of $\ff$
in ${\bK}^d \times {\bP}^{d-1}$. The blowing-up of ${\bK}^d$ with center $O$ is the proper birational morphism given
by the projection on the first factor $\s : Bl_0 {\bK}^d:=\overline{\G}_{\ff} \prox {\bK}^d$.
The exceptional fiber over $O$ is a rational divisor $B_0 \simeq {\bP}^{d-1}$.\vs
\begin{defn} \label{defn:const}
A {\em constellation of infinitely near points} (in short, a constellation) is a set
$\C=\{Q_0, \ldots , Q_n\}$, with $Q_i \geq Q_0 \in {X}_0 = X$,
such that
\( Q_i \in Bl_{Q_{i-1}} X_{i-1} =: X_{i} \stackrel{\s_{i-1}}{\longrightarrow} X_{i-1}\),
 for {\small $1 \leq i \leq n$};
where $Bl_{Q_{i-1}} X_{i-1}$ denotes the blowing-up of $X_{i-1}$  with center
$Q_{i-1}$ .
\end{defn}
The point $Q_0$ is called the {\em origin} of the constellation $\C$.
We call also the dimension of $X$ the {\em dimension} of $\C$. \vs
Let $\s_{\C} = \s_{0} \circ \dots \circ \s_{n} : X_{\C} \rightarrow X_{0}$
denote the composition of the blowing-ups of all the points of $\C$, where $\XC = X_{n+1}$.
Two constellations $\C$ and $\C'$ over $X$ are identified if there is an automorphism $\pi$
of X and an isomorphism $\pi' : X_{\C} \rightarrow X_{\C'}$ such that
$\s_{\C'} \circ \pi' = \pi \circ \s_{\C}$.
The relation $Q_j \geq Q_i$ is a partial ordering on the set of points of $\C$.
If this ordering is total, i.e. $Q_n \geq \cdots \geq Q_0$, we say that $\C$ is a
{\em chain} constellation. For example, for any constellation $\C$ and any $Q \in \C$, the set
$\C^Q := \{P \in \C \st Q \geq P\}$ of points preceding $Q$ is a chain constellation.
The number of points in $\C^Q$, different from $Q$, is called the {\em level} of $Q$.
The root of $C$ is the only point of level 0.
For each point $Q \in \C$ let $Q^+$ be the set of points of $\C$
{\em consecutive} to $Q$, i.e. the points following $Q$ for the ordering $\geq$ such that there is no strict
intermediate point; write $|Q^+|$ for the cardinal of this set. If $Q^+$ has only one
point, it denotes this point. \vs
For each point $Q = Q_i$, let $B_Q$ (or $B_i$) be the {\em exceptional} divisor $\s_{i}^{-1}(Q)$
on $X_{i+1}$, and $E_Q$ (or $E_i$) its successive {\em strict} (or {\em proper}) transforms on
any $X_j$ (which will be specified if necessary) with $Q_j \ge Q_i$, in particular in $X_{\C}$.
The {\em total} transforms are denoted by $E^*_Q$ or $E^*_i$.
The sets of divisors
$\{E_Q \st Q\in \C\}$ and $\{E^*_Q \st Q \in \C\}$, considered in $\XC$, are two basis of
the lattice $N^1 = \bigoplus_{Q \in \C} \bZ E_Q \cong \bZ^{n+1}$ of divisorial cycles with
{\em exceptional} support in $\XC$.
\begin{defn}
A point $Q_j \ge Q_i$ is {\em proximate} to $Q_i$ if $Q_j \in E_i$ in $X_j$;
notation : $Q_j \prox Q_i$ (or $j \prox i$).
The {\em proximity index} of a point $Q_j$ is defined as the number $ind(Q_j)$ of points
in $\C$ approximated by $Q_j$, i.e. $ind(Q_j) := \# \{Q_i \in \C \st Q_j \prox Q_i\}$.
\end{defn}
 If $R \in Q^+$ then $R \prox Q$, these are the so called trivial proximities.
If $R$ belongs to the intersection of several exceptional divisors produced by
blowing-up precedent points then $R$ is proximate to all these points. Since the irreducible
exceptional divisors we consider have normal crossing, in dimension $d$ a point may be proximate to at most
$d$ points. If the dimension of $C$ is {\em at least three}, then $R \prox Q$ if and only $R \ge Q$
and $E_R \cap E_Q \ne \emptyset$ in $\XC$. \vs
 Note that if $R \prox Q$ then $R \ge  Q$ , but the converse does not hold in general.
The proximity relation $(\prox)$ is a binary relation on the set of
points of a constellation, but not an ordering in general.
\begin{rem}
For each point $Q_i$, the only irreducible exceptional divisors, besides $E_i$,
appearing in the total transform $E_i^*$, in $\XC$, are exactly those
produced by blowing-up the points proximate to $Q_i$. Therefore
$E_i = E_i^* - \sum_{j \prox i} E_j^*$ . The so called {\em proximity matrix}
$((p_{ji}))$, with $p_{ii} = 1$, $p_{ji} =-1$ if $j \prox i$ and 0 otherwise,
is the basis change matrix from the $E_i$'s to the $E_j^*$'s
\end{rem}
\begin{defn}
The (proximity) {\em Enriques diagram} or {\em P-Enriques diagram} of a constellation $\C$
is the rooted tree $\GC$ equipped with a binary relation
$(\pprox)$, whose vertices are in one to one correspondence with
the points of $\C$, the edges with the couples of points $(R, Q)$
such that $R \in Q^+$, the root with the origin of $\C$, and the
relation $(\pprox)$ with the proximity relation $(\prox)$.
\end{defn}
 The Enriques diagram codes the chronology and incidence data of the
points in a constellation.
Any (finite) rooted tree (without the supplementary data of a binary
relation) may be the support graph of an Enriques diagram.
\begin{rem}
There is another graph in {\em dimension two} that may be associated to the (normal crossing) family of
irreducible exceptional divisors obtained by blowing-up the points of a constellation.
This is the so called {\em dual graph},
whose vertices are in bijection with the divisors and each edge is associated to the
intersection point of two divisors. The graph supporting the Enriques
diagram of a constellation and the dual graph  of the exceptional divisors may be quite
different. For instance the first one may be a chain but not the other, or viceversa.
\end{rem}
 A natural question is how to characterize the Enriques diagrams, i.e. which rooted
trees equipped with a binary relation on the set of vertices are induced by some constellation.
 Given a rooted tree $\G$, denote  by $(\succeq)$ the natural partial ordering on the set
$\V(\G)$ of its vertices : $p \succeq q$ if $q$ belongs to the chain from $p$ to the root;
similarly, if $(\pprox)$ is a binary relation on $\V(\G)$, let
$ind(q) = \# \{p \in \V(\G) \st q \pprox p\}$.
 For each vertex $q$, let $q^+$ be the set of consecutive vertices to $q$ with respect
to the ordering $(\succeq)$.
\begin{thm}   \label{thm:PEd}
Let $\G$ be a finite rooted tree equipped with a binary relation
$(\pprox)$ on the set of its vertices. Then $\G$ is the Enriques diagram of
a constellation of infinitely near points $\C$ and
$(\pprox)$ is induced by the proximity relation on $\C$ if and
only if , for any vertices $p, q, r$ of $\G$, the following
conditions are satisfied:
\begin{itemize}
\item[(a)]  $q \pprox p  \hs \Longrightarrow \hs q \succeq p$ , $q \ne p$
\item[(b)]  $q \in p^+ \hs \Longrightarrow \hs q \pprox p$
\item[(c)]  $r \succeq p \succeq q$ and $r \pprox q \hs \Longrightarrow \hs p \pprox q$
\end{itemize}
If these conditions hold, then the minimum dimension $d_{\PP}$ of a constellation
whose Enriques diagram is the given one is at most
$max (2, max_{q \in \V(\G)} (ind(q)) + 1)$.
\end{thm}
\begin{proof}
The necessity of the conditions follows easily.
For the sufficiency, proceed by induction on the number $| \V(\G) |$ of vertices.
If $| \V(\G) | > 1$, let $r$ be a maximal vertex of $\G$, and assume that a
constellation $\C'$ of dimension $d$ works for $\G' = \G \setminus \{r\}$.
Let $r \in q^+$, and $Q$ be the point of $\C'$ corresponding to $q$.
The set $Y := \{P \in \C' \st r \pprox p\}$ is contained in ${\C'}^Q$ by (a)
and $Q \in Y$ by (b).
By (c) one has $Q \prox P$ for each $P \in Y \setminus \{Q\}$, so that
$Q \in F := \bigcap_{P \in Y, P \ne Q} E_P$. It follows that $F \ne \emptyset$
and dim$(F) = d + 1 - | Y |$, by the normal crossing of the divisors $E_P$,
and on the other hand $ind(Q) \geq | Y \setminus \{Q\} | = | Y | - 1$.
Now, we need a point $R$ (in $X_{\C'}$) corresponding to $r$,
having the corresponding proximities, i.e.  a point
$R \in B_Q \bigcap F$ but not in $(Q^+ \bigcup _{P \in \C^Q \setminus Y} E_P )$ .
Such a point exists if $d \ge max_{p\in \V(\G)} ind(p) + 1$ (and at least $2$),
which is not less than $max_{p\in \V(\G')} ind(p) + 1$ so the inductive hypothesis
applies. This number is attained.
\end{proof}
\begin{rem}
The minimum dimension $d_{\PP}$ of constellations inducing a given P-Enriques diagram
may be one less than in the general case if there are no two maximal vertices $r$, with
maximum indices, say
$r_1$ and $r_2$, both in $q^+$, such that $ind(r_i) = ind(q) + 1$. Precisely, the
minimum dimension is
$d_{\PP}=max (\; 2\; ,\; max_{q \in \V(\G)} (ind(q) + t(q))\;)$,
where $t(q)=0$ (resp. $t(q)=1$) if
$s(q) := \#\{r\in q^+ \st ind(r) > ind(q)\} \le 1$
(resp. if $s(q) \geq 2$).
\end{rem}
\vsd

\subsection{Finitely supported ideals and idealistic
clusters}\label{sect:22}
\begin{defn}
{\rm
A {\it cluster} is a pair ${\mathcal K}=(\C, \underline{m})$ where
$\C=\{Q_0,\ldots,Q_n\}$ is a constellation and
$\underline{m}=(m_0,\ldots,m_n)$ is a sequence of
integers. The integer $m_i$ is called the {\it weight} (or {\it
virtual multiplicity}) of $Q_i$ in the cluster. }
\end{defn}
Given a cluster $\K$ as above, we can associate to it the following
divisor in $X_{\C}$ with exceptional support: $D(\K):=\sum_{i=0}^n
m_i E_i^*$. Hence, given a constellation $\C$, the choice of a
weight sequence $\underline{m}$ is equivalent to the choice of a
divisor in the semigroup $\sum_{i=0}^n \bZ_{\geq 0} E_i^*$.
\begin{defn}
{\rm Given a closed point $Q_0\in X$, an ideal $I$ in
$R_{Q_0}:=\OO_{X,Q_0}$ is {\it finitely supported} if $I$ is primary
for the maximal ideal $M_{Q_0}$ of $R_{Q_0}$ and there exists a
constellation $\C$ of infinitely near points of $X$ such that
$I\OO_{X_{\C}}$ is an invertible sheaf. An infinitely near point $P$
of $Q_0$ is a {\it base point} of $I$ if $P$ belongs to the
constellation with the minimal number of points with the above
property. We shall denote by $\C_I$ the constellation of base points
of $I$. }
\end{defn}
Given a finitely supported ideal $I$ in $R_{Q_0}=\OO_{X,Q_0}$, with
associated constellation of base points $\C_I=\{Q_0,\ldots,Q_n\}$,
we can associate to it a cluster $\K=\K_I=(\C_I,\underline{m})$, called
{\it cluster of base points of $I$}, as we shall describe now.
For any point $Q_i$, $0\leq i\leq n$, consider the chain
constellation of preceding points
$\C^{Q_i}=\{P_0=Q_0,P_1,\dots,P_r=Q_i\}$; the {\it weak transforms}
$I_{P_j}$ of $I$ at the points $P_j$ are defined by induction on $r$
by setting $I_{Q_0}=I$ and, for $i>0$, $I_{P_i}$ is the ideal in the
local ring $(R_{P_i}, M_{P_i})$ given by $$(x)^{-{\rm
ord}_{P_{i-1}}(I_{P_{i-1}})} I_{P_{i-1}} R_{P_i},$$ where,  ${\rm
ord}_{P_{i-1}} (I_{P_{i-1}}):=\max \{n\mid I_{P_{i-1}}\subseteq
M_{P_{i-1}}^n\}$ and $x$ is a generator of the principal ideal
$M_{P_{i-1}}R_{P_i}$.
For any $i$, $0\leq i\leq n$, the weight $m_i$ is defined to be
${\rm ord}_{Q_i}(I_{Q_i})$. Notice that the ideal $I_{Q_i}$ is
finitely supported and $m_i>0$. Moreover, it follows by induction
the following equality between ideal sheaves on $X_{\C_I}$:
$$ I\OO_{X_{\C_I}}=\OO_{X_{\C_I}}(-D(\K_{I})).  $$
\begin{rem}
{\rm The completion (or integral closure) $\overline{I}$ of a
finitely supported ideal $I$ is again finitely supported and
$\K_{\overline{I}}=\K_{I}$ (see \cite[Prop. 1.10]{L2}).
}
\end{rem}
For a fixed constellation $\C$ rooted at $Q_0\in X$, we shall denote
by $\J_{\C}$ the set of of finitely supported complete ideals $I$ of
$\OO_{X,Q_0}$ such that $\C_I\subseteq \C$. This set $\J_{\C}$ can
be  endowed with an operation, called $*$-product: given $I_1,I_2\in
\J_{\C}$, $I_1* I_2$ is defined to be the integral closure of the
product ideal $I_1 I_2$. Notice that $(\J_{\C},
*)$ has structure of commutative semigroup.
\begin{defn}
{\rm We shall say that a cluster $\K=(\C,\underline{m})$ is {\it
idealistic} if there exists a finitely supported ideal $I$ in
$R_{Q_0}$ such that $I\OO_{X_{\C}}=\OO_{X_{\C}}(-D(\K))$. Notice
that this implies that $I\in \J_{\C}$ and that $m_i$ is the weight
of $Q_i$ in $\K_{I}$ if $Q_i\in \C_{I}$ and $m_i=0$ otherwise. The
{\it galaxy} of $\C$ will be the set ${\mathcal G}_{\C}$ of
idealistic clusters on $\C$.
}
\end{defn}
From \cite[Sect. 18]{L1} it follows a characterization of the
idealistic clusters:
\begin{prop}
A cluster $\K=(\C,\underline{m})$ is idealistic if and only if
$\underline{m}\not=0$ and $-D(\K)$ is $\sigma_{\C}$-generated, i.e.
$\OO_{X_{\C}}(-D(\K))$ is generated by its global sections on a
neighbourhood of the exceptional fiber of
$\sigma_{\C}:X_{\C}\rightarrow X$.
\end{prop}
As a consequence of this proposition, one has that the galaxy
${\mathcal G}_{\C}$ of a constellation $\C$ has a natural structure
of commutative semigroup with the following operation: if
$\K_i=(\C,\underline{m}_i)\in {\mathcal G}_{\C}$, $i=1,2$,
$\K_1+\K_2:=(\C,\underline{m}_1+\underline{m}_2)$. Moreover, if
$I_1,I_2 \in \J_{\C}$, it is satisfied that $\K_{I_1 *
I_2}=\K_{I_1}+\K_{I_2}$. Also, Proposition 1.10 of \cite{L2} shows
that, given a constellation $\C$ and an idealistic cluster $\K \in
{\mathcal G}_{\C}$, there exists a unique finitely supported
complete ideal $I_{\K}\in \J_{\C}$ such that
$I_{\K}\OO_{X_{\C}}=\OO_{X_{\C}}(-D(\K))$; actually, it is the stalk
at the root of $\C$ of the sheaf ${\sigma_{\C}}_*
\OO_{X_{\C}}(-D(\K))$.
If we set $\mathbb{E}_{\C}^{\sharp}$ the semigroup of effective
divisors $D$ on $X_{\C}$ with exceptional support such that
$D\not=0$ and $\OO_{X_C}(-D)$ is $\sigma_{\C}$-generated, above
considerations are summarized  in the following result:
\begin{prop}\label{prop:2}
Given a constellation $\C$, the assignments $\K \mapsto D(\K)$ and
$\K \mapsto I_{\K}$ give isomorphisms of commutative semigroups
(${\mathcal G}_{\C},+) \rightarrow (\mathbb{E}_{\C}^{\sharp},+)$ and
$({\mathcal G}_{\C},+)\rightarrow (\J_{\C},*)$ respectively. The
inverse maps are defined by the assignments $D\mapsto \K_J$ (where
$J$ denotes the stalk of ${\sigma_{\C}}_* \OO_{X_{\C}}(-D)$ at the
origin of $\C$) and $I\mapsto \K_I$, respectively.
\end{prop}
\begin{rem}
{\rm Note that, in the above statement, for each ideal $I\in
\J_{\C}$ we are identifying the cluster $\K_I=(\C_I,\underline{m})$
with $(\C,\underline{m}')$, where $m'_i=m_i$ if $Q_i\in \C_I$ and
$m_i'=0$ otherwise. }
\end{rem}
Let ${\rm Nef}(X_{\C}/X)$ be the semigroup of non zero
$\sigma_{\C}$-nef divisors on $X_{\C}$ (also called either {\it
numerically effective} or {\it semiample} divisors), that is, those
exceptional divisors $D\not=0$ such that $D\cdot C\geq 0$ for any
exceptional curve (i.e. effective exceptional irreducible 1-cycle)
$C$ on $X_{\C}$.
\begin{prop}\cite[Prop. 1.22]{CGL}\label{prop:23}
If $\C$ is a constellation over $X$ then
$\mathbb{E}_{\C}^{\sharp}\subseteq -{\rm Nef}(X_{\C}/X)$.
\end{prop}
If the dimension $d$ of $X$ equals 2, then the effective exceptional
irreducible 1-cycles of $X_{\C}$ are the strict transforms of the
exceptional divisors. Then, a divisor $-D=-\sum_{i=1}^n m_i E_i^*$
is $\sigma_{\C}$-nef if and only if $-D\cdot
E_i=m_i-\sum_{j\rightarrow i} m_j \geq 0$ for $0\leq i\leq n$. This
inequalities are classically known as {\it proximity inequalities}
(see \cite{EC}, Chap. II, book 4). But, in this case, it is also
known that if $-D$ is $\sigma_{\C}$-nef then it is
$\sigma_{\C}$-generated \cite{EC,L1,Ca,L3}. Therefore, in dimension
2, the inclusion given in the statement of Proposition \ref{prop:23}
is an equality. Hence, we have a satisfactory description of the
idealistic clusters: a cluster is idealistic if and only if its
weights satisfy the proximity inequalities.
If $d>2$ the inclusion given in Proposition \ref{prop:23} need not
be an equality. This fact is shown in the following example, taken
from \cite{CGL}:
\begin{ex}
{\rm Let $X$ be a 3-dimensional non singular variety and let
$\C=\{Q_0,\ldots,Q_9\}$ a constellation consisting of a closed point
$Q_0\in X$ and nine points $Q_1,\ldots Q_9$ in general position on a
non singular cubic curve $C_0$ in the exceptional divisor $B_0$
(i.e. such that $C_0$ is the unique cubic curve in $B_0$ passing
through the nine points). Consider the divisor on $X_{\C}$ given by
$D=3E^*_0+\sum_{i=1}^9 E^*_i$. $-D$ is $\sigma_{\C}$-nef because, if
$C$ is any curve in $B_0$, the inequality $3\deg(C)-\sum_{i=1}^9
e_{Q_i}(C)\geq 0$ ($e_{Q_i}(C)$ denoting the multiplicity of $C$ at
$Q_i$) is obvious if $C=C_0$ and it follows from B\'ezout's theorem
otherwise. However, $-D$ is not $\sigma_{\C}$-generated because, if
otherwise, $C_0$ should be a fixed curve of the finitely supported
ideal $I$ such that $I\OO_{X_{\C}}=\OO_{X_{\C}}(-D)$. }
\end{ex}
However, the semigroup $\mathbb{E}_{\C}^{\sharp}$ has the property
to be closed under adjoints. In fact, if $K_{X_{\C}/X}$ is the
relative canonical divisor of the morphism $\sigma_{\C}$ and
$\K=(\C,\underline{m})$ is an idealistic cluster, then the adjoint
ideal $J_{\K}:={\sigma_{\C}}_* \OO_{X_{\C}}(-D(\K)+K_{X_{\C}/X})$
of the ideal $I_{\K}$ is again a finitely supported on $\C$
complete ideal, namely the one associated to the cluster with
weights $\max(0,m_i-d+1)$ for any $i$. This statement is due to
Lipman. Hence, one deduces the following result:
\begin{prop}\cite[Th. 3.3]{L4}\label{proplipman}
For a given constellation $\C$ one has that if $\sum_{i=0}^n m_i
E_i^* \in \mathbb{E}_{\C}^{\sharp}$ then $\sum_{i=0}^n
\max(0,m_i-d+1) E_i^* \in \mathbb{E}_{\C}^{\sharp}$.
\end{prop}
Moreover, a given divisor $\sum_{i=0}^n m'_i E_i^*$ is the
associated divisor to the adjoint of some finitely supported ideal
if and only if $\sum_{i=0}^n (m'_i+d-1) E_i^* \in
\mathbb{E}_{\C}^{\sharp}$. This follows from the definition and
above result. For $d=2$ this fact was proved in \cite[Th.
1]{hyry}.
\begin{defn}
{\rm A finitely supported complete ideal $I$ of a local ring
$\OO_{X,Q_0}$ is said to be $*$-{\it simple} if it cannot be
factorized as $*$-product of two proper ideals of $\OO_{X,Q_0}$ or,
equivalently,  $I$ is not the $*$-product of two proper ideals
belonging to ${\J}_{\C}$, whenever $I\in \J_{\C}$ for a
constellation $\C$. }
\end{defn}
\begin{rem}
{\rm
If $d=2$ the product of complete ideals is a complete ideal and,
hence, the operation $*$ coincides with the usual product of ideals;
in this case, the $*$-simple complete ideals are called {\it simple}
complete ideals.
 }
\end{rem}
In \cite{L2} Lipman associates, to each point $Q_j$ of a
constellation $\C=\{Q_0,\ldots,Q_n\}$, the unique finitely supported
complete $*$-simple ideal $\PP_{Q_j}$ of $R_{Q_0}$ whose cluster of
base points $\K_{\PP_{Q_j}}=(\C_{\PP_{Q_j}},\underline{m})$
satisfies the conditions: $\C_{\PP_{Q_j}}=\C^{Q_j}$, the weight of
$Q_j$ equals 1 and the weight sequence $\underline{m}$ is minimal
for the reverse lexicographical ordering in $(\mathbb{Z}_{\geq
0})^{\ell+1}$, where $\ell$ is the level of $Q_i$. For simplicity of
notation, we shall denote by $D(Q_j)$ the divisor on $X_{\C}$ given
by $D(\K_{\PP_{Q_j}})$, that is:
$$D(Q_j):=\sum_{Q_i\leq Q_j} m_{ij} E_i^*,$$
where $m_{ij}$ is the virtual multiplicity of $Q_i$ in the cluster
$\K_{\PP_{Q_j}}$. Since $m_{jj}=1$ for all $j$, one has that the set
$(D(Q_0),\ldots, D(Q_n))$ is a basis of $N^1$ and the basis change
matrix from $(D(Q_i))$ to $(E_i^*)$ is the matrix
$\mathbf{M}_{\C}:=((m_{ij}))$.
As a consequence of this fact and Proposition \ref{prop:2} we get
the Lipman's unique factorization theorem (see \cite{L2}):
\begin{thm}
Given a constellation $\C=\{Q_0,\ldots,Q_n\}$, for each $I\in
{\mathcal J}_{\C}$ we can write {\it formally}, in a unique form,
the ideal $I$ as $*$-product of the $*$-simple ideals ${\mathcal
P}_{Q_i}$ associated with the points in $\C$:
\begin{equation}\label{fact}
I=\prod_{0\leq i \leq n}^* {\mathcal P}_{Q_i}^{r_i}
\end{equation}
with $r_i\in \bZ$ for all $i=0,\ldots,n$. Moreover, the vector
$\underline{r}=(r_1,\ldots,r_n)$ can be computed as
$\underline{r}^t=\mathbf{M}_{\C}^{-1} \underline{m}^t$, where
$\K_{I}=(\C,\underline{m})$.
\end{thm}
\begin{rem}
{\rm Notice that, in the statement above, $r_i=0$ if $Q_i\not\in
\C_{I}$. Moreover, the expression (\ref{fact}) (with non necessarily
positive exponents) means that there exists a $*$-product of $I$
times ideals ${\mathcal P}_{Q_i}$ which is equal to a $*$-product of
ideals ${\mathcal P}_{Q_j}$, which distinct factors in both sides of
the equality.
 }
\end{rem}
If $d=2$ the situation is very simple because of Zariski's theory of
complete ideals (see \cite{Z} and \cite{ZS}). In this case, there
exists unique factorization of complete ideals as product of simple
complete ideals. Moreover, the exponents $r_i$ are non-negative.
Lipman, in \cite{L3}, provides a modern presentation of Zariski's
results. The matrix $\mathbf{M}_{\C}$, in this case, coincides with
the inverse of the transpose of the proximity matrix $\mathbf{P}_{\C}$
and $(D(Q_i))$ is the dual $\bZ$-basis of $(-E_i)$ with respect to the
bilinear pairing $N^1\times N^1\rightarrow \bZ$ given by the
intersection product.
The sub-semigroup ${\mathcal L}_{\C}$ of $\J_{\C}$ of those ideals
which are $*$-products of the ideals ${\mathcal P}_{Q_i}$ with
non-negative exponents is nothing but the free semigroup generated
by the ${\mathcal P}_{Q_i}$. By the isomorphisms in Proposition
\ref{prop:2}, it corresponds to the sub-semigroup ${\mathcal
L}{\mathcal G}_{\C}$ of ${\mathcal G}_{\C}$ generated by the
clusters $\K_{{\mathcal P}_{Q_i}}$ and to the sub-semigroup
$\mathbb{L}_{\C}$ of $\mathbb{E}_{\C}^{\sharp}$ generated by the
divisors $D(Q_i)$, $0\leq i\leq n$.
\subsection{Idealistic clusters and embedded
resolutions}\label{sect:23}
The objective of this section is, on the one hand, to define several
concepts whose aim is to describe how an effective divisor in $X$
passes through the infinitely near points involved by a cluster and,
on the other hand, to state a result showing that, if the
characteristic of the ground field $\bK$ is 0, then the morphism
$\sigma_{\C}$ associated to the constellation of base points of a
finitely supported ideal $I$ can be seen as the embedded resolution
of a subvariety defined by {\it general enough} elements of $I$. The
above mentioned concepts will help us to precise the meaning of {\it
general enough}. Recall that a projective birational morphism
$\pi:Z\rightarrow Y$ is an {\it embedded resolution} of a reduced
subvariety $V$ of $Y$ having an isolated singularity at $Q_0\in X$
if $Z$ is non singular, $\pi$ induces an isomorphism of $Z\setminus
\pi^{-1}(Q_0)$ to $X\setminus \{Q_0\}$ and $\pi^{-1}(V)$ is a normal
crossing subscheme.
Fix $\C=\{Q_0,\ldots,Q_n\}$ a constellation over $X$ with origin at
$Q_0$ and set $S:={\rm Spec}(\OO_{X,Q_0})$ and
$S_{\C}:=X_{\C}\times_X S$. We shall denote also by $\sigma_{\C}$ to
the induced morphism $S_{\C}\rightarrow S$. The constellation $\C$
can be naturally regarded as a constellation over $S$ with origin at
its closed point $Q_0$ and $\sigma_{\C}=\sigma_0\circ \cdots
\sigma_{n}:S_{\C}=S_{n+1}\rightarrow \cdots \rightarrow
S_1\rightarrow S_0=S$ being its associated composition of
blowing-ups.
\begin{defn}\label{proper}
{\rm
Let $\K=(\C,\underline{m})$ be a cluster, with $\C$ as above. Let
$D$ be an effective divisor on $S$.
\begin{itemize}
\item[(a)] For $1\leq i\leq n$, the divisor on $S_{i}$ given by $
\check{D}_i:= (\sigma_0\circ \cdots \circ
\sigma_{i-1})^*D-\sum_{j=0}^{i-1} m_i E_i^*$ is called the {\it
virtual transform} of $D$ on $S_{i}$ with respect to the cluster
$\K$. The virtual transform of $D$ on $S$, $\check{D}_0$, will be
considered to be $D$. \item[(b)] $D$ is said to {\it pass} (resp.
to {\it pass effectively}) (resp. to {\it pass properly}) through
$\K$ if for any $J=\{i_1<\cdots <i_k\}$ with $k=1$ (resp. $k=1$)
(resp. $1\leq k\leq d$) such that $E_{J}:=E_{i_1}\cap\cdots \cap
E_{i_k}\subseteq S_{\C}$ is not empty, the multiplicity at
$Q_{i_k}$ of the inverse image ${D}_{J}$ of $\check{D}_{i_k}$ on
$E_{i_1}\cap\cdots \cap E_{i_{k-1}}\subseteq S_{i_k}$ (or
$S_{i_k}$ if $k=1$) is $\geq$ (resp. $=$) (resp. $=$)
$m_{J}:=m_{i_k}$.
\end{itemize}
}
\end{defn}
If $D$ and $\K$ are as above and $D$ passes properly with respect to
$\K$, we denote the projective tangent cone to ${D}_{J}$ at
$Q_{i_k}$ by $TC(D)_{J}$. This is a hypersurface of degree $m_{J}$
in $B_{J}:=E_{i_1}\cap \cdots \cap B_{i_k}\cong {\bP}^{d-k}$.
\begin{prop}\cite[Prop. 3.4]{CGL}\label{prop:24}
With the notations of Definition \ref{proper}, the map which takes
$D$ to $\check{D}_{n+1}$ (the virtual transform on $S_{n+1}=S_{\C}$)
is a one to one correspondence between the set of effective divisors
in $S$ which pass through $\K$ and the complete linear system
$|-D(\K)|$ on $S_{\C}$. Moreover, for any effective divisor $D$ in
$S$:
\begin{itemize}
\item[(a)] $D$ passes effectively
through $\K$ if and only if, for any $Q_i\in \C$, the multiplicity
of the strict (or proper) transform of $D$ at $Q_i$ is $m_i$.
\item[(b)] If $D$ passes properly through $\K$ then, for any $J$ as
in Definition \ref{proper}:
\begin{itemize}
\item[(i)] the subvariety $E_{J}$ on $S_{\C}$ is not contained in
the strict  transform $\tilde{D}$ of $D$,
\item[(ii)] for $1\leq k<d-1$, the scheme $E_{J}\cap \tilde{D}$ is
the strict transform by $\sigma_{J}:E_{J}\rightarrow B_{J}$ of
$TC(D)_{J}$ and for any $i\rightarrow J$ (i.e. $i\rightarrow
i_{\ell}$, $1\leq \ell \leq k$), the multiplicity at $Q_i$ of the
strict transform of $TC(D)_{J}$ is $m_i$.
\end{itemize}
\end{itemize}
\end{prop}
Given an element $f\in R_{Q_0}=\OO_{X,Q_0}$ we denote by $H_f$ the
hypersurface in $S$ defined by $f$.
\begin{defn}
{\rm A $r$-uple $(f_1,\ldots,f_r)$ of elements in $R_{Q_0}$ with
$1\leq r<d$ is said to be {\it non degenerated} with respect to a
cluster $\K=(\C,\underline{m})$ if the hypersurfaces
$H_{f_1},\ldots,H_{f_r}$ pass properly through $\K$ and, for any $J$
such that $\dim E_{J}\geq 1$, the hypersurfaces
$\{TC(H_{f_i})_J\}_{i=1}^j$ of $B_J$ intersect transversally except
maybe at proper points of $B_J$ in $\C$. }
\end{defn}
\begin{prop}\cite[Prop. 3.6]{CGL} If $(f_1,\ldots,f_r)$ is non
degenerated with respect to $\K$, then
$\sigma_{\C}:S_{\C}\rightarrow S$ is an embedded resolution of the
subvariety of $S$ defined by $f_1,\ldots,f_r$.
\end{prop}
\begin{thm}
If the characteristic of the ground field $\bK$ is 0, $I$ is a
finitely supported ideal of $R_{Q_0}$ and $\C$ is its constellation
of base points $\C_{I}$, then the above morphism
$\sigma_{\C}:S_{\C}\rightarrow S$ is an embedded resolution of the
subvariety of $S$ defined by $r$, $1\leq r<d=\dim X$, general
elements in $I$.
\end{thm}
\begin{proof}
It follows from the preceding proposition and the fact that, since
the characteristic of $\bK$ is 0, a $r$-uple of general elements of
$I$ is non degenerated with respect to $\K_I$ \cite[Prop. 3.8]{CGL}.
\end{proof}
\subsection{Characteristic cones and factorization
properties}\label{sect:24}

As we have already seen, the results in dimension 2 concerning
unique factorization of complete ideals as a product of simple
complete ideals do not extend to higher dimensions. The use of {\it
characteristic cones} provides an interesting framework to study
factorization properties of complete ideals in dimension greater
than 2. The main objective of this section is to provide an overview
of this fact. To begin with, we shall define some convex cones
related to a projective morphism, providing also some basic
properties. Afterwards, we shall consider the particular case in
which such a morphism is the one associated with a constellation.
Let $f: V\rightarrow Y$ be a projective morphism between
algebraic schemes over $\bK$. Denote by $N_1(V/Y)$
(resp. $N^1(V/Y)$) the free abelian group of $1$-dimensional
cycles on $V$ whose support contracts (by $f$) to a closed point
in $Y$ (resp. Cartier divisors on $V$) modulo numerical
equivalence. Recall that a 1-dimensional cycle $C$ (resp. a
Cartier divisor $D$) is numerically equivalent to $0$ iff $D\cdot
C=0$ for all Cartier divisors $D$ (resp. all integral curves $C$
contracted to a closed point of $Y$) on $V$. Intersection theory
provides a $\bZ$-bilinear pairing $N^1(V/Y) \times
N_1(V/Y)\rightarrow \bZ$ which extends to a $\bR$-bilinear pairing
$A^1(V/Y) \times A_1(V/Y)\rightarrow \bR$, where
$A^1(V/Y):=N^1(V/Y)\otimes_{\bZ} \bR$ and
$A_1(V/Y):=N_1(V/Y)\otimes_{\bZ} \bR$. The dimension $\rho(V/Y)$
of $A^1(V/Y)$ is finite and
the above intersection pairing makes $A^1(V/Y)$ and $A_1(V/Y)$
dual vector spaces \cite[Chap. IV, Sect. 4]{K}. For simplicity of notation, given a contracted
effective curve $C$ (resp. a Cartier divisor $D$) on $V$, its
classes in $N_1(V/Y)$ and $A_1(V/Y)$ (resp. $N^1(V/Y)$ and
$A^1(V/Y)$) will also be denoted by $C$ (resp. $D$).
Let $NE(V/Y)$ be the {\it cone of curves} of $V$ relative to $f$,
that is, the convex cone in $A_1(V/Y)$ generated by the classes of
effective contracted curves in $V$. Denote by $P(V/Y)$ the {\it nef
cone} relative to $f$ (also called {\it semiample cone}), that is,
the dual cone of $NE(V/Y)$ or, equivalently, the convex cone in
$A^1(V/Y)$ consisting of vectors $x$ such that $x\cdot C\geq 0$ for
every contracted effective curve in $V$. According to \cite[Chap.
IV, Sect. 4]{K}, the cone $P(V/Y)^o\cup \{0\}$ ($P(V/Y)^o$ being the
topological interior of $P(V/Y)$) is generated by the classes of the
relatively ample divisors $D$ (this means that, for every coherent
sheaf ${\mathcal F}$, the canonical map $f^*f_*{\mathcal F}\otimes
\OO_V(mD)\rightarrow {\mathcal F}\otimes \OO_V(mD)$ is surjective
for all $m$ sufficiently large or, equivalently, $Y$ is covered by
affine subsets $U$ such that the restriction of $D$ to $f^{-1}(U)$
is ample).
The {\it characteristic cone} relative to $f$,
$\tilde{P}(V/Y)$, is defined to be the convex cone of $A^1(V/Y)$
generated by the classes of Cartier divisors $D$ such that the
natural sequence $f^*f_*\OO_{V}(D)\rightarrow \OO_{V}(D)\rightarrow
0$ is exact,
The inclusion $\tilde{P}(V/Y)\subseteq P(V/Y)$ is clear. Moreover,
since some multiple of an ample divisor is generated by global
sections, it follows that $P(V/Y)^o\subseteq \tilde{P}(V/Y)$ and
hence $P(V/Y)^o=\tilde{P}(V/Y)^o$. Notice that, since $f$ is
projective, there exist relatively ample divisors and, therefore,
the dimension of both cones $P(V/Y)$ and $\tilde{P}(V/Y)$ is
$\rho(V/Y)$. When $Y={\rm Spec}(\bK)$, the above defined spaces and
convex cones are denoted by $A_1(V)$, $A^1(V)$, $NE(V)$, $P(V)$ and
$\tilde{P}(V)$ respectively. Assume now that  $S={\rm
Spec(\OO_{X,Q_0})}$, with $X$ and $Q_0$ as in the preceding
sections, and $\C=\{Q_0,\ldots,Q_n\}$ is a constellation over $S$
with associated composition of blowing-ups
$\sigma_{\C}:S_{\C}\rightarrow S$. In this case $N^1(S_{\C}/S)$
coincides with the free abelian group ${\mathbb E}_{\C}$ and
$\{E_1,\ldots,E_n\}$ is a $\bR$-basis of $A^1(S_{\C}/S)$ \cite[Lem.
15]{Cu}. Moreover, the characteristic cone $\tilde{P}(S_{\C}/S)$
(resp. nef cone $P(S_{\C}/S)$) is the one generated by the image in
$A^ 1(S_{\C}/S)$ of the divisors $D$ such that $-D$ (resp. $D$)
belongs to ${\mathbb E}_{\C}^{\sharp}$ (resp. ${\rm
Nef}(S_{\C}/S)$). If $d=\dim S=2$, $N^1(S_{\C}/S)$ and
$N_1(S_{\C}/S)$ are identified with $N^1$. Taking into account the
unique factorization of the ideals in $\J_{\C}$ as a product of the
simple complete ideals ${\mathcal P}_{Q_i}$ (by Zariski's theory)
and Proposition \ref{prop:2}, one has that the semigroup ${\mathbb
E}_{\C}^{\sharp}$ is freely generated by the divisors $D(Q_i)$. This
implies that the cone $\tilde{P}(S_{\C}/S)$ is the {\it regular}
cone (which coincides with $P(S_{\C}/S)$) generated by the images in
$A^1(S_{\C}/S)$ of the divisors $-D(Q_i)$. For $d>2$, the cone
$\tilde{P}(S_{\C}/S)$ contains the regular sub-cone $L_{\C}$
generated by the divisors $-D(Q_i)$ but, in general, one has
$L_{\C}\not=\tilde{P}(S_{\C}/S)$. If $d>2$, the cone
$\tilde{P}(S_{\C}/S)$ is not, in general, regular (as we shall see
later) and, hence, there is not, in general, unique factorization of
finitely supported complete ideals as $*$-product of $*$-simple
ideals. Furthermore, the regularity of the characteristic cone does
not imply unique factorization of complete ideals (see \cite[Example
4.2]{CG}). There is a weaker notion than the unique factorization
which is detected from the structure of the characteristic cone: the
{\it semi-factoriality}.
\begin{defn}
{\rm Let $G$ be a commutative semigroup with cancellation law. An
element $g\in G\setminus \{0\}$ is called {\it extremal} if $g$ has
no inverse in $G$ and if a factorization (additively written)
$ng=a+b$ (with $n$ an integer) implies that $sa=qg$ and $tb=pg$ for
suitable integers $a,b,p,q$. Two extremal elements $x$ and $y$ are
called {\it equivalent}, $x\sim y$, if there are positive integers
$m$ and $n$ such that $nx=my$. $G$ is {\it semi-factorial} if to
each $g\in G$ with $g\not=0$ there is an integer $n>0$ such that
$ng$ is a sum of extremal elements, and this factorization is unique
in the following sense: if $ng= a_1+\ldots+a_s$, $a_i$ extremal,
$a_i \not\sim a_j$ for $i\not=j$, and $mg=b_1+\ldots+b_t$, $b_i$
extremal, $b_i \not\sim b_j$ if $i\not=j$, then $s=t$ and $a_i\sim
b_i$ after reindexing. }
\end{defn}
Notice that an ideal $I\in \J_{\C}$ is extremal in $(\J_{\C},*)$ iff
$D(\K_{I})$ is extremal in ${\mathbb E}_{\C}^{\sharp}$ iff
$-D(\K_{I})$ generates an extremal ray of the cone
$\tilde{P}(S_{\C}/S)$. From this fact, it can be easily deduced the
following result:
\begin{prop}
The semigroup $\J_{\C}$ is semi-factorial if and only if the cone
$\tilde{P}(S_{\C}/S)$ is simplicial (that is, it is spanned by
linearly independent elements).
\end{prop}
The following result and the examples mentioned below show that, in
general, the semi-factoriality of $\J_{\C}$ does not hold if $d>2$.
\begin{prop}\cite[Th. 20]{Cu} Suppose that $S={\rm Spec}(R)$, where $R$ is the localization
at $(x,y,z)$ of the polynomial ring $\mathbb{K}[x,y,z]$. Let $Q_0$
be the closed point of $S$ and let $\C_n=\{Q_0,Q_1,\ldots,Q_n\}$ be
a constellation over $S$ such that $Q_1,\ldots,Q_n$ are $n$ closed
points in general position on the exceptional divisor associated to
the blowing-up at $Q_0$. Then $\tilde{P}(S_{\C_n}/S)$ is simplicial
if and only if $n\leq 2$.
\end{prop}
 There are several examples in the literature showing that the characteristic cone $\tilde{P}(S_{\C}/S)$ can have very
different shapes, indicating the existence of different
factorization's phenomena:
\begin{itemize}
\item[(i)] It can be polyhedral (that is, finitely generated) but not simplicial \cite[Example
4.1]{CG}.
\item[(ii)] It can have infinitely many extremal rays \cite[Example
2]{Cu}.
\item[(iii)] It can be non-closed (\cite[Example 4.3]{CG} and \cite[Example 3]{Cu}).
\item[(iv)] As we have pointed out before, it can be regular but with
$\J_{\C}$ not having unique factorization \cite[Example 4.2]{CG}.
\end{itemize}
\subsection{Toric constellations}\label{sect:25}
Now we consider the toric constellations and proximity. We begin by recalling some
definitions and fixing notations for toric varieties (for a detailed treatment see
some of the basic references on this subject, e.g. chapter 1 of \cite{O} or \cite{TE}).
\vs
Let $N \cong {\bZ}^d$ be a lattice of dimension $d\ge 2$ and $\Si$ a {\em fan}
in $N_{\bR} = N {\otimes}_{\bZ} {\bR}$, i.e. a finite set of
strongly convex rational polyhedral cones such that every face of a cone of $\Si$ belongs
to $\Si$ and the intersection of two cones of $\Si$ is a face of both.
Denote by $X_{\Si}$ the toric variety
over a field $\bK$ associated with $\Si$, equipped with the action of an algebraic
torus $T \cong {(\bK^*)}^d$. There is a one to one canonical correspondence
between the $T$-orbits in $X_{\Si}$ and the cones of $\Si$. Two basic facts
of this correspondence are that the dimension
of a T-orbit is equal to the codimension of the corresponding cone, and that a T-orbit
is contained in the {\em closure} of another T-orbit if and only if  the cone associated
with the first one contains the cone associated with the second one.
The {\em morphisms} of toric varieties are the equivariant maps induced by the maps of fans
$\varphi : (N', \Si') \prox (N, \Si)$ such that $\varphi : N' \prox N$ is a $\bZ$-linear
homomorphism whose scalar extension $\varphi : N'_\bR \prox N_\bR$ has the property
that for each $\s' \in \Si'$ there exists $\s \in \Si$ such that $\varphi(\s')\subset \s$ ;
(see \cite{O}, 1.5).\vs
Let $X_0 := X_{\Si_0} \cong {\bK}^d$ be the $d$-dimensional affine toric variety
associated with the fan $\Si_0$ formed by all the faces of a regular d-dimensional
rational cone $\De$ in $N_{\bR}$. Recall that a rational cone is called {\em regular} (or
nonsingular) if the primitive integral extremal points form a subset of a basis of the lattice.
\vs
A {\em toric constellation} of infinitely near points is a constellation
$\C = \{Q_0, \ldots, Q_n\}$ such that each $Q_j$ is a fixed point for the
action of the torus in the toric variety $X_j$ obtained by blowing-up $X_{j-1}$
with center $Q_{j-1}$, {\small $1 \le j \le n$}. If a toric constellation is a chain,
it is called a {\em toric chain}.
The identification of constellations stated after definition \ref{defn:const} is the same in the toric case,
with equivariant isomorphisms.
\vsd

\noindent {\em Codification of toric constellations and proximity}.
\vsd

By choosing a fixed {\em ordered basis} $\B = \{v_1,...,v_d\}$ of the lattice $N$  we obtain
a {\em codification} of the toric constellations, as well as criteria for proximity
and (as shown in the following) linear proximity.\vs
Let $\De = \langle \B \rangle$ be the (regular) cone generated by the basis $\B$.
The blowing-up $\s_i : X_i \rightarrow X_{i-1}$ of the closed orbit $Q_{i-1}$, is
described as an elementary subdivision of a fan, as follows.\vs
The variety $X_1$ is the toric variety associated with the fan $\Si_1$, obtained as
the minimal subdivision of $\Si_0$ which contains the ray through $u = \sum_{1\le j\le d}  v_j$.\\
\noindent For each integer $i$, $1 \le i \le d$, let $\B_i$ be the ordered basis of $N$
obtained by replacing $v_i$ by $u$ in the basis $\B$; and let
$\De_i := \langle \B _i \rangle$. The exceptional divisor $B_0$ is the closure
in $X_1$ of the $T$-orbit defined by the ray through $u$, and each T-fixed point in
$X_1$ corresponds to a maximal cone $\De_i$ of the fan $\Si_1$, {\small $1 \le i \le d$}.
The choice of the point $Q_1\geq Q_0$ is thus equivalent to the choice of an
integer $a_1$, $1\le a_1\le d$, which determines a cone $\De_{a_1}$ of
the fan $\Si_1$.
The subdivision $\Si_2$ of $\Si_1$
corresponding to the blowing-up of  $Q_1$ is obtained by replacing
$\De_{a_1}$ (and its faces) in $\Si _1$
by the cones $\De_{a_{_1}i}:=\langle \B_{a_{_1}i}\rangle $
(and their faces), where $\B_{a_{_1}i}$ is the ordered basis of $N$
obtained from $\B_{a_1}$ by the substitution of its $i$-th vector
by $\sum_{v\in \B_{a_1}}v$ .
The choice of $Q_2\in B_1$ is equivalent to the choice of an integer $a_2$ ,$
1\le a_2\le d$, which determines a (regular) cone $\De_{a_{_1}a_{_2}}$.\vs
Proceeding by induction on $n$ we obtain a {\em codification} of toric chains and
also constellations, since for each $Q\in \C$ , the constellation ${\C}^Q$ is a chain.
The codification is given by trees with weighted edges, where the weights
are integers $a$, $1\le a\le d$, which give the {\em direction} in which the
following blowing-up is done. The precise description follows.
\begin{defn}
Let $\G$ be a tree, $\E(\G)$ the set of edges of $\G$, d an integer, d $\geq 2$.\\
\noindent A {\em d-weighting} of $\G$ is a map $\ga : \E(\G) \rightarrow \{1, \ldots, $d$\}$
which associates to each edge of $\G$ a positive integer not greater than d,
such that two edges with a common origin have different weights. A couple $(\G ,\ga)$
is called a \em{d-weighted tree}.
\end{defn}
\begin{prop}
Let $\B$ be  an ordered basis of the lattice $N$ and $n$ a positive integer.
\begin{itemize}
\item[(a)] The map which associates to each sequence of integers
$\{a_1, \ldots, a_n \}$ such that $1\le a_i \le d$, $1\le i \le n$, the toric
chain $\{Q_0, \ldots , Q_n\}$
where $Q_0$ is the $T$-orbit corresponding to the cone $\De =\langle \B \rangle$,
and where $Q_i$, {\small $1\le i\le n$}, is the $T$-orbit in $X_i$ corresponding to the
cone $\De_{a_1 \ldots a_i}$ of the fan $\Si_i$, is a bijection between the set of
such sequences and the set of d-dimensional toric {\em chains} with $n+1$ points.
\item[(b)]  A natural bijection between the set
of d-dimensional toric constellations and the set of d-weighted trees
is induced by the correspondence {\em (a)}.
\end{itemize}
\end{prop}
\begin{rem}
Note that in a d-weighted tree each vertex is the origin of at most d edges.
A d-weighting of a tree $\G$ induces a partition of the set $\E(\G)$ of edges,
where two edges are in the same class if they have the same weight.
To each class of isomorphism of d-dimensional toric constellations is associated
a unique class of isomorphism of trees equipped with a partition of the set of
edges, partition with at most d classes of edges \cite{GP}.
\end{rem}
\
Given a toric constellation by a d-weighted graph, a vertex
following $q$ through a chain with edges weighted by a sequence
$(a_1, \ldots , a_k)$  is denoted by $q(a_1, \ldots , a_k)$ ; if $Q$ is the
point corresponding to $q$, then the point corresponding to $q(a_1, \ldots , a_k)$
is written in a similar way  $Q(a_1, \ldots , a_k)$.
\begin{prop} \label{prop:critprox}
({\em Criterion} for proximity in terms of a codification)
$Q(a_1, \ldots , a_k) \prox Q$ if and only if $a_1 \ne a_j$ for $2\le j\le k$ .
\end{prop}
\begin{proof}
The criterion follows from the fact that this is the condition to obtain, by elementary
subdivisions of a regular fan, an adjacent maximal cone $\De_{a_1 \ldots a_k}$
(corresponding to a 0-dimensional orbit) to the central ray of $\De_{a_1}$
(corresponding to the exceptional divisor) of the first subdivision
of the cone $\De$ corresponding to $Q$. This is equivalent to saying that
$Q(a_1, \ldots , a_k) \in E_Q$, i.e. $Q(a_1, \ldots , a_k) \prox Q$.
\end{proof}
We obtain a characterization of {\em toric} P-Enriques diagrams and the minimum dimension
for a toric constellation with a given P-Enriques diagram ( \cite{GP}, \cite{G-S})
\newpage

\begin{thm}   \label{thm:toricPEd}
A P-Enriques diagram $(\G , (\pprox))$ is toric, i.e. may be induced by a toric constellation,
if and only if:
\begin{itemize}
\item[(a)] The proximity index is non-decreasing, i.e.
$ind(r) \geq ind(q)$ if $r \succeq q$.
\item[(b)] If $r$ is proximate to $q$, then there is at most one  vertex
$s$ consecutive to $r$ and not proximate to $q$, i.e. if $r \pprox q$ then
$\# \{ s\in r^+ \st s \not \pprox q\} \le 1$.
\end{itemize}
If these conditions hold, then the minimum dimension $dt_{\PP}(\G, (\pprox))$ of
a toric constellation inducing the given P-Enriques diagram $(\G, (\pprox))$ is
$max ( 2, max_{q\in \G} (ind(q) + s(q)))$,
where $s(q) := \#\{r\in q^+ \st ind(r) > ind(q)\}$ is the number of consecutive
points to $q$ whose proximity index is greater than the proximity index of q.
\end{thm}
\begin{rem}
The minimum dimension $dt_{\PP}$ may be greater than $d_{\PP}$, the dimension in the
not necessarily toric case (Theorem~\ref{thm:PEd}), because there are less points
available, so one needs to add $s(q)$ to the proximity index, not just $1$ as in
the general case.
\end{rem}
\begin{cor}
A P-Enriques diagram $(\G , (\pprox))$ whose graph $\G$ is a {\em chain},
is toric if and only if
the proximity index is not decreasing. In this case, the minimum dimension
of an associated constellation is the index of the terminal point
(and at least $2$).
\end{cor}
\begin{proof}
In the toric chain case the condition (b) of the theorem is automatically
satisfied and $max_{q\in \G} (ind(q) + s(q))) = max_{q\in \G} (ind(q))$ holds.
\end{proof}
\begin{exs}
(1)\hs The simplest example of a non-toric P-Enriques diagram is a chain with four
vertices, say  $q_0, q_1, q_2, q_3$ such that, besides the trivial proximities
of consecutive vertices, the only other proximity is $q_2 \prox q_0$.
In this example one has $ind(q_2)=2$ and $ind(q_3)=1$; condition (a) fails.\vs

\noindent(2)\hs Another example of a non-toric case is a graph of type $\bD_4$, with a
non-central vertex as the root, and with only the proximities of consecutive vertices.
In this case condition (b) fails. Remark that both cases may be induced by two dimensional constellations.\vs

\noindent(3)\hs If the central vertex is the root in a graph of type $\bD_n$, with $n\ge 4$,
and if the only proximities are those of consecutive vertices, then conditions
(a) and (b) hold; the minimal dimension
of a constellation inducing this P-Enriques diagram is $n-1$ for toric constellations
and two for non-toric ones. If $q_0$ is the root, then $ind(q_0) = 0$, $s(q_0) = n-1$,
$t(q_0) = 1$, and $ind(q) = 1$, $s(q) = 0$, $t(q)=0$  for each $q \ne q_0$.\vs

(See Figures (1), (2) and (3)).
\end{exs}

\begin{figure}
\begin{picture}(10,6)(0,0)
\put(0,0){\begin{picture}(2,6)(0,0)
\put(1,1){\line(0,1){3}} \multiput(1,1)(0,1){4}{\som}
\put(1.1,0.9){$q_0$} \put(1.1,1.9){$q_1$}
\put(1.1,2.9){$q_2$} \put(1.1,3.9){$q_3$}
\put(0.8,2){\oval(0.8,2)[l]} \put(0.4,2.1){\vector(0,-1){0,2}}
\put(.5,0){Figure (1).} \end{picture}}

\put(3,0){\begin{picture}(4,6)(0,0)
\put(2,1){\line(0,1){1}} \put(2,2){\line(-1,1){1}}
\put(2.1,0.9){$q_0$} \put(2.1,1.9){$q_1$}
\put(1.15,2.9){$q_2$} \put(3.15,2.9){$q_3$}
\put(2,2){\line(1,1){1}} \multiput(2,1)(0,1){2}{\som}
\put(1,3){\som} \put(3,3){\som} \put(1.5,0){Figure (2).}
\end{picture}}

\put(7,0){\begin{picture}(4,6)(0,0)
\put(2,1){\line(-1,1){1}}
\put(2,1){\line(1,1){1}} \put(3,2){\som}
\multiput(2,1)(0,1){2}{\som} \put(2,1){\line(0,1){1}}
\put(1,2){\som}
\put(2.1,0.9){$q_0$} \put(0.5,2){$q_1$} \put(3.1,2){$q_{n-1}$}
\put(1.3,1.9){$\cdots$} \put(2.3,1.9){$\cdots$}
\put(1.5,0){Figure (3).} \end{picture}}
\end{picture}
\end{figure}
\vsd

\noindent {\em Linear proximity and characteristic cones}
\vsd

In dimension two the exceptional divisors appearing in the definition of the
proximity relations are (rational) curves. In higher dimension we
introduce, in the toric case, a condition involving curves which will be
finer, in general, than the proximity. This new condition arises naturally
for toric clusters in higher dimension, from the generalization of the proximity
inequalities.
\begin{defn}
Let $\C =\{Q_0, \ldots, Q_n\}$ be a toric constellation. A point $Q_j$
is {\em linear proximate} to a point $Q_i$ with respect to a one dimensional
T-orbit $\ell \subset B_{i}$ if $Q_j$ belongs to the strict transform in $X_j$
of the closure of $\ell$.
This relation is denoted by $Q_j \lprox Q_i$ , or $Q_j \llprox Q_i$ if we need
to specify the line $\ell$ involved.
\end{defn}
If $R \lprox Q$ then $R \prox Q$, but the converse does not hold in general.
\begin{prop}   \label{prop:critLP}
(Criterion for linear proximity)
Let $Q$ be a point in a toric constellation of dimension d. Each 1-dimensional orbit
$\ell$ in the exceptional divisor $B_Q$ contains in its closure only two fixed points,
say $Q(a)$ and $Q(b)$.
Then $R \lprox Q$ if and only if there are integers $a$, $b$ and $m$ such that $a\ne b$,
$1\le a\le d$, $1\le b\le d$, $0\le m$ and $R=Q(a, b^{[m]})$ or
$R=Q(b, a^{[m]})$, where $x^{[m]}$ means $x$ repeated $m$ times.
\end{prop}
\begin{proof}
The wall running between the cones corresponding to $Q(a)$ and $Q(b)$ is the cone
corresponding to the line defined by this two points in $B_Q$. The only maximal cones,
obtained by elementary subdivisions, having this wall as a face are those corresponding
to the points $Q(a, b^{[m]})$ or $Q(b, a^{[m]})$ for some $m\ge 0$.
\end{proof}
In dimension two, proximity and linear proximity are equivalent. One implication
may be generalized for toric {\em chains} in any dimension.
\begin{prop}
If $\C$ is a toric chain (in any dimension), the proximity relation
determines the linear proximity relation.
\end{prop}
\begin{proof}
If $R \lprox Q$, then $P \prox Q$ for any $P$ such that $R\ge P\ge Q$, $P\ne Q$,
and these are the only proximities, for the intermediate points in the chain
from $Q$ to $R$, besides the proximities of consecutive points. Conversely,
assuming this property, then $R \llprox Q$ for the line $\ell$ determined by the
point $Q^+$ and the direction $Q^{++}$ in the projective space $B_Q$, if $Q^{++}$
is defined and precedes $R$, or any line through $Q$ otherwise. Indeed, this
assumption forces the code of $R$ to be $Q(a, b^{[m]})$
for some weights $a$ and $b$, $m\ge 0$.
\end{proof}
Note that in general the linear proximity does not determine the proximity,
even for chains.
\vsd

\noindent {\em LP-Enriques diagrams.}
\vsd

We introduce now some definitions leading to the notion of the so called
(linear proximity) LP-Enriques diagrams. This is a LP generalization,
for toric constellations of dimension higher than two, of the Enriques diagrams
of two dimensional constellations. We will give later an application of these diagrams
to prove a converse Zariski theorem. \vs
Given a rooted tree $\G$, a sub graph formed by two chains with a common
root and no common edge is called a {\em bi-chain}. \vs
If $\G$ is the rooted tree associated with a toric constellation $\C$, $q$
the vertex corresponding to $Q \in \C$ and $\ell$ is a 1-dimensional
orbit in $B_Q$, then $\G_q (\ell)$ denotes the full subgraph of $\G$
with vertices corresponding to Q and to the points $R\in \C$ such that
$R \lprox Q$. Let $\G(q)$ be the family of the maximal $\G_q(\ell)$
when $\ell$ describes the set of one dimensional orbits in $B_Q$.
A vertex $q\in \G$ is called {\em simple} (resp.\ {\em ramified}) if $|q^+| = 1$
 (resp.\ if $| q^+ | > 1$).
The following properties are easily checked with the linear proximity criterion
(Proposition~\ref{prop:critLP}).
\begin{prop}  \label{prop:LPS}
Let $\C$ be a toric constellation, $\G$ the associated tree.
\begin{itemize}
\item[1.]
  \begin{itemize}
   \item[(a)] For each $q\in \G$, the family $\G(q)$ is non-empty and the elements
    of $\G(q)$ are chains or bi-chains with root $q$.
   \item[(b)] If $\g$, $\g' \in \G(q)$ and $\g \subset \g'$, then $\g = \g'$.
  \end{itemize}\vs
\item[2.]
   \begin{itemize}
   \item[(a)] Two distinct elements of $\bigcup_q \G(q)$ have at most one common edge.
   \item[(b)] Two edges with common ramification root vertex $q$ (resp.\  the edge
    with the simple root vertex $q$) belong (resp.\ belongs) to one and only one
    element of $\G(q)$.
  \end{itemize}\vs
\item[3.]
   \begin{itemize}
   \item[(a)] For each $q\in \G$ and $r\in q^+$ there is at most one vertex $s\in r^+$
    such that the chain $(q, r, s)$ is not contained in any element of $\G(q)$.
   \item[(b)] If $(p, \ldots , q, r)$ is a chain contained in a $\g \in \G(p)$ and
    $s \in r^+$ satisfies 3.(a), then the chain $(p, \ldots, q, r, s)$ is
    contained in $\g$.
   \end{itemize}
\end{itemize}
\end{prop}
\begin{defn}
The {\em LP-Enriques diagram} of a toric constellation $\C$ is the associated graph
$\GC$ equipped with the linear proximity structure formed by the family of full subgraphs
$\{\GC(q) \st q\in \GC\}$.
\end{defn}
We obtain a characterization of LP-Enriques diagrams and the minimum dimension
for a toric constellation with a given LP-Enriques diagram ( \cite{GP}, \cite{G-S}).
\begin{thm} \label{thm:DPL}
The couple $(\G, \{ \G(q) \st q\in \G\})$, given by a tree $\G$ and a family of full
subgraphs $\G(q)$, is the LP-Enriques diagram of a toric constellation $\C$ if and
only if the properties $1$, $2$ and $3$ hold.\vs
The minimum dimension of the constellations with given LP-Enriques diagram
is $d_{\PP\LL} = max( 2, max_{q\in \G} ( |q^+| + n_q ))$, where
$n_q=max_{r\in q^+} \#\{\g \in \G(q) \st r\in \g$ and $\g$ is a chain of length $> 1 \}$
\end{thm}
\begin{rem}
A LP-Enriques diagram may be induced by two non-isomorphic constellations.
In some cases, for instance if for each vertex q the family $\G(q)$ has only
bi-chains or is reduced to the vertex, then the constellation inducing the
given LP-Enriques diagram is unique (up to isomorphism of constellations),
and its dimension is $|q_0^+|$ if $q_0$ denotes the root.
The maximum possible linear proximity dimension $d_{LP}$ of a fixed tree, by changing
its LP structure, is the number of edges. In this case all the chains (resp. bi-chains)
have only one edge (resp. two edges) or are reduced to a vertex, for the maximal ones.
\end{rem}
\vsd

\noindent {\em Characteristic cones of toric constellations}.
\vsd

For toric constellations the characteristic cone may be explicitly obtained
(see~\cite{CGL}, theorem 2.10). Note that in this case the characteristic cone
coincides with the semiample cone (see~\cite{TE}, page 47).
The natural ideals to consider are the invariant ideals for the toric
action, so that the constellations of base points are toric.
The conditions that such an ideal $\I$ is finitely supported and complete
are formulated in terms of the Newton polyhedron $\N$ of $\I$ relative to the local system
of parameters of the local ring, induced by a basis of the lattice where
the fan lives.
The first condition is that the fan associated to the Newton polyhedron
(which gives the normalized blowing-up of center $\I$) admits a regular
subdivision obtained by elementary subdivisions of the regular cone $\De$ corresponding to
$Q_0$; and the second  one is that every monomial corresponding to an integral
point of $\N + \De^\vee$ is in $\I$, where $\De^\vee$ denotes the dual cone
of $\De$.\vs
The following result generalizes, for toric constellations in any dimension,
the two dimensional proximity inequalities found by Enriques.
Recall Proposition~\ref{prop:critLP}.
\begin{thm}  \label{thm:tcc}
Let $\C$ be a toric constellation of dimension $d$.\\
\noindent The characteristic cone
associated with $\C$ is the cone generated by the classes of
the divisors \hs $D_{\um} = \sum_{Q \in \C} m_Q E^*_Q$ \hs such that $\um$ verifies
the linear proximity inequalities \hs $m_Q \ge \sum_{P \llprox Q} m_P$ \hs
for each $Q \in \C$ and each $\ell = \ell (Q(a), Q(b))$, $a\ne b$ $1\le a\le d$,
$1\le b\le d$.\vs
\end{thm}
\begin{proof}
The linear proximity inequalities are necessary, since they are equivalent
to $(D_{\um} \cdot \bar\ell) \le 0$ for a semiample divisor $-D_{\um}$ and the
closure $\bar\ell$ of each one dimensional orbit $\ell (Q(a), Q(b))$.
Conversely, if these inequalities hold, then $-D_{\um}$ is semiample since the classes
of the closures of the one dimensional orbits generate the cone of the
numerically effective curves $NE$, and then the divisor is $\s$-generated because
$\s$ is a toric morphism.
\end{proof}
\begin{rem}
$ $
\begin{itemize}
\item[(1)] A constructive proof giving the Newton polyhedron of the unique
complete ideal associated to such a divisor $D_{\um}$ (or the
corresponding {\em idealistic cluster}) is presented in~\cite{CGL}
theorem 2.10 (ii).
\item[(2)] From Theorem \ref{thm:tcc} and Proposition \ref{proplipman}
one can characterize, in numerical terms, which toric clusters
correspond to adjoint ideals of finitely supported ideals. In fact,
such toric clusters $(\C,\underline{m}')$ are exactly those such
that the toric cluster given by $(\C, \underline{m})$, where
$m_Q=m'_Q+(d-1)$, satisfies the conditions of Theorem \ref{thm:tcc}.
\item[(3)]  Theorem \ref{thm:tcc} has been recently used by A.
Lemahieu and W. Veys in \cite{lem} to describe the zeta functions
for non degenerated hypersurfaces with respect to 3-dimensional
toric clusters and prove the monodromy conjecture for them.
\end{itemize}
\end{rem}
\begin{cor}  \label{cor:tch}
We keep the notations of the theorem. Let $\C=\{Q_0, \ldots , Q_n\}$ be a
toric {\em chain}.
\begin{itemize}
\item[(a)] The characteristic cone associated with $\C$ is given by
\begin{center}
 $m_i \ge \sum_{j \lprox i} m_j$ , {\small $0\le i \le n$}.
\end{center}
\item[(b)] The divisor $D_n = \sum_{0\le i \le n} m_{i,n} E^*_i$ associated
to the special $*$-simple ideal $\PP_{Q_n}$ is given by  $m_{n,n} =
1$, $m_{i,n} = \sum_{j \lprox i} m_{j,n}$, for {\small $0\le i\le
n$}.

\end{itemize}
\end{cor}
\begin{proof}
($a$) follows from the Theorem and the fact that for each point there is only one
relevant inequality, since $\C$ is a chain.
($b$) follows from ($a$) since the minimality property of $\um$ is
obtained if $m_{n,n}=1$ and if every inequality involving an index $i\neq n$
becomes an equality.
\end{proof}
The special $*$-simple ideals, and the exponents of the factorizations are determined
by the linear proximities:
\begin{thm}
Let $\C$ be a toric constellation.
\begin{itemize}
\item[(a)] Let $(D_Q)_{Q\in \C}$ be the basis of $N^1$ corresponding to the
special $*$-simple ideals with base points in $\C$. Then $D_Q=
\sum_{P\in \C} m_{P Q} E^*_P$, where $m_{PQ}= 0$ if $P\not\le Q$,
$m_{QQ}=1$ and $m_{PQ}=\sum_{R\in \C \st Q\ge R\lprox P} m_{RQ}$ if
$P\le Q$.

\item[(b)] Let $\bP_L=((l_{PQ}))$ be the {\em linear proximity matrix}
defined by $l_{PP}=1$, $l_{PQ}=-1$ if $P\lprox Q$ and $0$ otherwise.
Then $^{t}\bP_L$ is the basis change matrix from $(E^*_Q)$ to
$(D_Q)$.

\item[(c)] Let $\I$ be a toric finitely generated ideal with base
points in $\C$. Then the exponents of its factorization in terms of
special $*$-simple ideals are: $r_Q=m_Q-\sum_{P\lprox Q} m_P$.

\end{itemize}
\end{thm}
\begin{proof}
($a$) follows from  corollary ~\ref{cor:tch}, ($b$).
($b$) and ($c$) follow from ($a$) and linear algebra.
\end{proof}
Recall the definition of the LP structure of the tree $\G$
associated  with $\C$ (Proposition~\ref{prop:LPS}).
\begin{cor}\label{cor:23}
Let $P_\C=P(X_\C/ X)$ be the characteristic cone associated with $\C$.
The following conditions are equivalent:
$(a)$ The cone $P_\C$ is regular.
$(b)$ $(D_Q)_{Q\in \C}$ is a basis of the semigroup $P_\C \bigcap N^1$.
$(c)$ The cone $P_\C$ is simplicial.
$(d)$ The special $*$-simple factorizations have only non negative exponents.
$(e)$ For each $Q\in \C$ there is only one (maximal) chain or bichain in $\G(q)$.
\end{cor}
\begin{proof}
The conditions ($a$), ($b$), ($c$) and ($d$) are equivalent since the divisors $D_Q$ form a basis
of $N^1$. The  equivalence between ($e$) and ($c$) follows from the preceding theorem,
and the fact that the supporting hyperplanes of the maximal faces of the cone $P_\C$
are those associated with the maximal elements of $\G_Q$ for each $Q\in \C$.
\end{proof}
\begin{rem}  \label{rem:tch}
In particular, every toric chain constellation in any dimension has a regular
characteristic cone. There are also non-chain constellations with this property.
\end{rem}
We give now an application of the LP-Enriques diagrams for a converse Zariski
Theorem for toric constellations.
Recall the notations and results on the minimal LP-dimension  $d_{\LL \PP}$ of a
LP-Enriques diagram (Theorem~\ref{thm:DPL}).
\begin{thm}\label{teo:28}
The characteristic cone of a toric constellation is {\em regular} if and only if
its LP-Enriques diagram is induced by a two dimensional constellation.
\end{thm}
\begin{proof}
The characteristic cone of any two dimensional constellation is regular, by Zariski.
Conversely, assume that the characteristic cone is regular. Then $\G(q)$ has only one
element for each $q \in \G$, by the last Corollary. It follows necessarily
that $0\le |q^+|\le 2$.
Now, $0\le |q^+| \le1$ implies that $0\le n_q\le 1$ and $|q^+|=2$ implies that $n_q=0$.
It follows that the minimal dimension $d_{\LL \PP}$ of a constellation inducing
the given LP-Enriques diagram is two.
\end{proof}
In the general toric case, the characteristic cone $P_{\C}$ contains
the regular sub-cone $L_{\C}$ which is given by the inequalities
$m_Q\geq \sum_{P\lprox Q} m_P$ for $Q\in \C$. Notice that conditions
in Corollary \ref{cor:23} are also equivalent to the equalities
$\mathbb{L}_{\C}=\mathbb{E}_{\C}^{\sharp}$ or $L_{\C}=P_{\C}$.
For a non toric two dimensional constellation $\C$, Zariski's theory
shows that one also has $\mathbb{L}_{\C}=\mathbb{E}_{\C}^{\sharp}$
and that $L_{\C}=P(X_{\C}/X)=\tilde{P}(X_{\C}/X)$ is a regular cone.
For $d=2$, Zariski's theory asserts also further properties as the
following ones:
\begin{itemize}
\item[(i)] If $I\in {\mathcal J}_{\C}$ and $Q\in \C$, then the weak
transform $I_Q$ is a complete ideal; so $I_Q\in {\mathcal
J}_{\C^Q}$.
\item[(ii)] If $I,J\in {\mathcal J}_{\C}$ then $I*J=IJ$, so $IJ$ is
a complete ideal.
\item[(iii)] If $\K \in {\mathcal G}_{\C}$ is given by
$\K=(\C,\underline{m})$ then $\mu(I_{\K})=m_0+1$, where $\mu$ stands
for the minimal  number of generators of an ideal.
\end{itemize}
The following result, due to E. Tost\'on \cite{toston}, shows that
above properties are also true for toric clusters in the
sub-semigroup ${\mathcal L}{\mathcal G}_{\C}$ of the galaxy
${\mathcal G}_{\C}$ (recall the notations used at the end of Section
\ref{sect:22}).
\begin{thm}\cite[Th. 3.1, Prop 3.2, 3.4]{toston}\label{tos} Let $\C$ be a toric
constellation. Then one has:
\begin{itemize}
\item[(i)] If $I\in {\mathcal L}_{\C}$ and $Q\in \C$ then the weak
transform $I_Q$ is a complete ideal and $I_Q\in {\mathcal
L}_{\C^Q}\subseteq {\mathcal G}_{\C^Q}$.
\item[(ii)] If $I,J\in {\mathcal L}_{\C}$ then $I*J=IJ$; so $IJ$ is
a complete ideal.
\item[(iii)] If $\K\in {\mathcal L}{\mathcal G}_{\C}$ is given by
$\K=(\C,\underline{m})$ then one has
$\mu(I_{\K})=\binom{m_0+d-1}{d-1}$.
\end{itemize}
\end{thm}
Part (iii) follows from (ii) and the fact that, in the toric case,
the conditions of effective passage through an idealistic cluster
are linearly independent. For it, (ii) is applied to the ideals
$I_{\K}$ and $M$, $M$ being the maximal ideal.
\begin{rem}
$ $
\begin{itemize}
\item[(1)] If $\C$ satisfies the equivalent conditions in Corollary
\ref{cor:23} the statement of Theorem \ref{tos} is true for the
whole galaxy of $\C$, as one has ${\mathcal G}_{\C}={\mathcal
L}_{\C}$ and ${\mathcal J}_{\C}={\mathcal L}_{\C}$. Thus, Zariski's
theory is fully extended for such constellations. In particular, it
is true for toric chains. This gives an additional insight to
Theorem \ref{teo:28}.
\item[(2)] If $\C$ does not satisfy conditions in Corollary
\ref{cor:23} then the statement of Theorem \ref{tos} is not longer
true for ideals in ${\mathcal G}_{\C}\setminus {\mathcal L}_{\C}$ as
the following examples, also due to E. Tost\'on in \cite{toston},
show.
\end{itemize}
\end{rem}
\begin{exs}
Consider $d=3$, coordinates $x,y,z$ and let $\C$ be the toric
constellation consisting of two chains of respective edge weights
given by $\{1,2,2\}$ and $\{3,2,1\}$. Let $\K$ be the cluster
supported on $\C$ with $\underline{m}$-weights given by $3,1,1,1$
and $3,2,1,1$ respectively on above chains. Then one has
$$I=I_{\K}=\langle x^4, x^3y, x^2z, x^2y^2,x^2yz, xy^3,xy^2z,xz^3,
y^3, y^2z^2,yz^3,z^5\rangle.$$ Now, if $Q$ is a $0$-dimensional
$T$-orbit connected to $Q_0$ by the edge of weight 3, then one has
$(\frac{x}{z}\frac{y}{z})^2\not\in I_{Q}$ but
$(\frac{x}{z}\frac{y}{z})^2\in \overline{I}_Q$. Thus (i) is not true
for $I$. On the other hand, one has $xy^2z\in I*M$ but $xy^2z\not\in
I M$, which shows that (ii) is not true for the ideals $I$ and $M$.
Finally one has $\mu(I)=\dim(I/IM)>\dim(I/I*M)=\binom{3+2}{2}=10$,
so (iii) is again not true for $I$.
Notice that $\K$ is an idealistic cluster. However $\K$ is not in
${\mathcal L}{\mathcal G}_{\C}$ as one has
$$3=m_{Q_0}<\sum_{P\lprox Q_0} m_P=2+1+1+1+1=6.$$
Even it is not true that $I^2$ is complete when $I$ is a toric
finitely supported complete ideal. This happens, for instance, if
$d=4$ and $I=I_{\K'}$, where $\K'=(\C',\underline{m})$ is given by
the constellation $\C'$ consisting of the three chains with edge
weights $\{1,2,3\}$, $\{3,2,1\}$, $\{4,2,2\}$ and respective
$\underline{m}$-weights given by $3,1,1,1$; $3,2,1,1$; $3,1,1,1$. If
$x,y,z,w$ are the coordinates, one has $x^2y^2zw\not\in I^2$ but
$x^2y^2zw\in \overline{I}^2$.
\end{exs}
\begin{rem}
The exact conditions under which the equality
$\mu(I_{\K})=\binom{m_0+d-1}{d-1}$ is true for clusters are
investigated in homological terms in \cite{cruz}.
\end{rem}
\section{Global theory}\label{sec:3}
\subsection{Configurations, global clusters and linear systems}
\begin{defn}
{\rm A {\it configuration of infinitely near points} (configuration,
in short) is a finite union of constellations whose origins are
closed points of $X$. }
\end{defn}
If $\C$ is a configuration, as in the case of constellations,
$\sigma_{\C}:X_{\C}\rightarrow X$ will denote the composition of the
blowing-ups of all the points in $\C$; moreover two configurations
$\C$ and $\C'$ over $X$ are identified if there exist an
automorphism $\pi$ of $X$ and an isomorphism $\pi' : X_{\C}
\rightarrow X_{\C'}$ such that $\s_{\C'} \circ \pi' = \pi \circ
\s_{\C}$. Given a configuration $\C$, the relation $\geq$ and the
proximity relation $\rightarrow$ between points of $\C$ are defined
as in the case of constellations. Also, for a point $Q$ in $\C$, we
define the concepts of associated {\it constellation $\C^Q$ of
preceding points}, {\it level} of $Q$, its {\it proximity index}
$ind(Q)$ and the set of {\it consecutive} points $Q^+$ as those
referred to the maximal constellation contained in $\C$ to which $Q$
belongs. The exceptional divisors, its strict and total transforms
and the lattice of divisorial cycles with exceptional support in
$X_{\C}$ will be denoted as in the case of constellations. Also, the
proximity matrix is defined, in the same way, as the basis change
matrix from the $E_i$'s to the $E^*_j$'s. If $\C=\bigcup_{i=1}^r
\C_i$, where the $\C_i$'s are disjoint constellations, we define the
{\it P-Enriques diagram} associated to $\C$ as the disjoint union of
the P-Enriques diagrams $\Gamma_{\C_1},\ldots,\Gamma_{\C_r}$. We
also adapt in the obvious way the concept of {\it cluster} to the
global case defining it as a pair $(\C,\underline{m})$, where
$\C=\{Q_i\}_{i=0}^n$ is a configuration and
$\underline{m}=(m_0,\ldots,m_n)$ is a sequence of integers; also,
$m_i$ is called the {\it weight} or {\it virtual multiplicity} of
$Q_i$ in the cluster. As in the local case, we can associate to the
cluster the divisor with exceptional support $D(\K):=\sum_{i=0}^n
m_i E_i^*$.
\begin{defn}
{\rm A sheaf of ideals $\I$ on $X$ is said to be {\it finitely
supported} if there exists a finite set $S$ of closed points of $X$
such that, for each closed point $Q\in X$, the stalk $\I_Q$ is a
finitely supported ideal in $\OO_{X,Q}$ (resp. $\I_Q=\OO_{X,Q}$)
whenever $Q\in S$ (resp. $Q\not\in S$). This implies that there
exists a configuration $\C$ over $X$ (whose roots are the points in
$S$) such that $\I \OO_{X_{\C}}$ is an invertible sheaf. The {\it
configuration of base points} of $\I$ (denoted by $\C_{\I}$) is the
configuration with the minimal number of points having this
property, that is, $\C_{\I}=\bigcup_{Q\in S} \C_{{\I}_Q}$.}
\end{defn}
Given a finitely supported ideal sheaf $\I$,
its {\it cluster of base points} $\K_{\I}$ will  be the cluster
$(\C_{\I},\underline{m})$ where $\C_{\I}=\{Q_0,\ldots,Q_n\}$ is
the above defined configuration of base points and,
for $0\leq i\leq n$, $m_i$ is the weight of $Q_i$ in the cluster
$\K_{{\I}_O}$, $O\in X$ being the image of $Q_i$ by
$\sigma_{\C_{\I}}$.
Notice that Definition \ref{proper} makes sense also for
configurations. Also, given a cluster $\K=(\C,\underline{m})$ and
given a non-empty complete linear system $|R|$ on $X$, the set of
elements of $|R|$ passing through $\K$ form a linear system
${\mathcal L}_{|R|}(\K)$ on $X$ whose elements are in one to one
correspondence with the complete linear system on $X_{\C}$ given by
$|\sigma_{\C}^*R-D(\K)|$. The correspondence is given by the
assignation, to each effective divisor $D\in |R|$, of its virtual
transform on $X_{\C}$ (defined in the same way as in Definition \ref{proper}).
Also, fixed a non-empty linear system $\varrho\subseteq |R|$, one
can consider the ideal sheaf $\I(\varrho)$ whose stalks at the
points of $X$ are generated by the local equations of the divisors
in $\varrho$. This ideal sheaf defines a closed sub-scheme of $X$,
called the {\it base point scheme} of $\varrho$. Notice that
$\I(\varrho)$ may not be finitely supported (resp. complete).
The determination of the dimension of a linear system of
hypersurfaces passing through a cluster is a very classical problem
in algebraic geometry, mainly considered when the configuration $\C$
consists of a set of (proper) points on a projective space $\bP^r$
in general position. It has been present in the works of B\'ezout,
Pl\"ucker, Cremona, M. Noether, Bertini, C. Segre, Castelnuovo,
Enriques, Severi and, more recently, Alexander, Hirschowitz,
Ciliberto, Miranda, Harbourne, among many others. There are several
conjectures and open questions on this problem and it is related
to other topics like the 14-th problem of Hilbert \cite{Na}, the
symplectic packing problem \cite{Mc,Xu} or the Waring's problem in
number theory (see Section 7 of \cite{cil3}). For interesting
surveys on this subject we refer the reader to \cite{Gim},
\cite{cil3}, \cite{Mir} or \cite{harb3}. However, we shall return to
it later, but focusing our attention on the case of linear systems
of plane curves.

\subsection{Cones of curves of rational surfaces and P-sufficient configurations}\label{sect32}
As we shall see along this paper, techniques related to the cone
of curves of a projective regular rational surface have
fruitful applications in several problems in Algebraic Geometry. So,
we include here a brief exposition of the basic properties of the
cone of curves, and also we recall the notion, introduced in
\cite{galmon} and \cite{galmon3}, of {\it P-sufficient
configuration} (over a relatively minimal rational surface). This
concept depends only on the P-Enriques diagram of the configuration
and it implies the polyhedrality of the cone of curves associated
with the surface obtained by blowing-up the points of the
configuration. The obtention of conditions implying the
polyhedrality of the cone of curves is an interesting issue, as we
shall see in the applications.
Let $Z$ be a regular projective rational surface and consider its
cone of curves $NE(Z)\subseteq A_1(Z)$ and its closure with respect
to the real topology, denoted by $\overline{NE}(Z)$. Notice that, in
this case, we can identify the spaces $A_1(Z)$ and $A^1(Z)$; we
shall denote them by $A(Z)$. We shall assume that $\dim A(Z)\geq 3$
(otherwise the cone of curves is regular).
 Recall that, if $C$ is a convex cone
of $A(Z)$, a {\it face} of $C$ is a sub-cone $F\subseteq C$ such
that $a+b\in F$ implies that $a,b\in F$, for all pair of elements
$a,b\in C$. The $1$-dimensional faces of $C$ are the {\it extremal
rays} of $C$.
Fix an ample divisor $H$ on $Z$. By Kleiman's ampleness criterion
\cite{K}, $H\cdot x>0$ for all $x\in \overline{NE}(Z)\setminus
\{0\}$ and, hence, the cone $\overline{NE}(Z)$ is strongly convex.
This implies that it is generated by its extremal rays. Consider the
cone
$$Q(Z)=\{x\in A(Z)\mid x^2\geq 0,\;\; H\cdot x\geq 0\}.$$
By the Hodge index theorem \cite[V.1.9]{har} there exists a basis of
$A(Z)$ for which the intersection bilinear form on $A(Z)$ is given
by the diagonal matrix ${\rm diag}(1,-1,\ldots,-1)$ in such a way
that $Q(Z)$ is defined by an inequality of the type $x_1\geq
(\sum_{i=2}^{\rho(Z)} x_i^2)^{1/2}$ in the suitable coordinates.
Then, $Q(Z)$ is the half-cone over an Euclidean ball of dimension
$\rho(Z)-1$, which is strictly convex. One has that $Q(Z)\subseteq
\overline{NE}(Z)$ \cite[II.4.12.1]{kollar1} and, therefore, the
extremal rays of $\overline{NE}(Z)$ must be spanned by elements
$x\in A(Z)$ such that $x^2\leq 0$. The extremal rays of
$\overline{NE}(Z)$ which are not in $Q(Z)$ are spanned by classes of
integral curves $C$ with $C^2<0$ \cite[II.4.12.3]{kollar1}.
Moreover, if $C$ is an integral curve on $Z$ such that $C^2<0$ then
$C$ generates an extremal ray of $\overline{NE}(Z)$
\cite[II.4.12.2]{kollar1}. The extremal rays of ${NE}(Z)$ generated
by elements $x$ such that $K_Z\cdot x<0$ ($K_Z$ being a canonical
divisor on $Z$) are known as a consequence of the Mori cone theorem
(see \cite[III.1]{kollar1} for instance): they are exactly those
spanned by the images in $A(Z)$ of the $(-1)$-curves (that is,
integral regular rational curves whose self-intersection is equal to
$-1$); furthermore, if there are infinitely many $(-1)$-curves, the
accumulation points of the set of generated extremal rays must be on
the orthogonal hyperplane  to the canonical class, $K_Z^{\perp}$.
However, very little is known concerning the region $NE(Z)\cap (K_Z
\cdot x\geq 0)$. From the classical theory of surfaces, it is
well-known that $Z$ can be obtained by blowing-up the points of a
configuration $\C$ over a relatively minimal rational surface $X$,
that can be either the projective plane $\bP^2$ of a Hirzebruch
surface $\mathbb{F}_a:=\bP(\OO_{\bP^1}\oplus \OO_{\bP^1}(a))$, $a$
being a non negative integer, $a\not=1$. We fix, from now on, both
$\C$ and $X$ such that $Z=X_{\C}$. When $K_Z^2\geq 0$ we have the
following results:
\begin{itemize}
\item[(i)] If $K_Z^2> 0$ then $NE(Z)$ is polyhedral (see
\cite{manin} and \cite[Cor. 1(i)]{galmon} for the case in which $X$
is the projective plane, and \cite[Th. 2(a)]{galmon3} for the
general case). Notice that this happens if and only if the
cardinality of $\C$ is $\leq 8$ (resp. $7$), whenever $X=\bP^2$
(resp. $X$ is a Hirzebruch surface).
\item[(ii)] If $K^2_Z=0$, then either $NE(Z)$ is polyhedral, or the
set of extremal rays of $\overline{NE}(Z)$ has a unique accumulation
point, which is spanned by $-K_Z$ (see \cite[Cor. 1(ii)]{galmon} and
\cite[Th. 2(b)]{galmon3}). Notice that $K_Z^2=0$ if and only if  the
cardinality of $\C$ is $9$ (resp. $8$), whenever $X=\bP^2$ (resp.
$X$ is a Hirzebruch surface).
\item[(iii)] If $K^2_Z=0$ and $K_Z\cdot D>0$ for some effective
divisor $D$ on $Z$, then $NE(Z)$ is polyhedral \cite[Th.
2(c)]{galmon3}.
\end{itemize}
The cone of curves is not, in general, polyhedral. For instance, in
\cite[Example 4.3]{CG} it is provided an example of a chain
constellation of 9 points over $X=\bP^2$ such that the surface
$X_{\C}$ has infinitely many $(-1)$-curves and, therefore, it has
infinitely many extremal rays. The following result provides
conditions for the polyhedrality and regularity of $NE(Z)$ and the
characteristic cone $\tilde{P}(X)$ in terms of the existence of
curves passing through a certain cluster.
\begin{prop}\cite[Sect. 2.2.2]{galmon3}\label{conic}
Set $\C=\{Q_0,\ldots,Q_n\}$ and consider the cluster
$\K=(\C,\underline{m})$, where $m_i=1$, $0\leq i\leq n$. Assuming
that $X=\bP^2$ the following properties hold:
\begin{itemize}
\item[(a)] If there exists a line passing through $\K$ then the cones
$NE(Z)$ and $P(Z)$ are regular, $NE(Z)$ being generated by
$E_0,\ldots,E_n$ and the image in $A(Z)$ of the strict transform of
the line.
\item[(b)] If there exists a conic $C$ passing through $\K$, then
$NE(Z)$ is a polyhedral cone generated by $E_0, \ldots, E_n$, the
images in $A(Z)$ of the strict transforms of the lines passing
through two points in $\C$, and the image of the virtual transform
of $C$ in $Z$ with respect to $\K$. Moreover $\tilde{P}(Z)$ is a
closed cone.
\item[(c)] If there exists a conic passing through $\K$ and $n\leq 3$ then
$NE(Z)$ and $\tilde{P}(Z)$ are regular cones.
\item[(d)] if $n\geq 4$ and there exists an integral conic such
that its successive strict transforms pass through $Q_0,\ldots,Q_4$,
then  $\tilde{P}(Z)$ is not simplicial.
\end{itemize}
Assuming that $X$ is a Hirzebruch surface $\mathbb{F}_a$, $NE(Z)$ is
polyhedral whenever a curve in the linear system $|(1-a)F+2M|$ pass
through $\K$, $F$ being a fiber of the natural morphism
$\mathbb{F}_a\rightarrow \bP^1$ and $M$ being the divisor of zeros
of a non-trivial global section of $\OO_{\mathbb{F}_a}(1)$.
\end{prop}
Now, we define the above mentioned notion of {\it P-sufficient}
configuration.
\begin{defn}
{\rm Set $\C=\{Q_0,\ldots,Q_n\}$ and consider the divisors $D(Q_i)$
such that ${\mathcal P}_{Q_i}\OO_{Z}=\OO_{Z}(-D(Q_i))$, ${\mathcal
P}_{Q_i}$ being the simple complete ideal of $\OO_{X,O}$ associated
with $Q_i$, $O$ being the image of $Q_i$ on $X$ (see Section
\ref{sect:22}). Consider the $(n+1)$-dimensional symmetric matrix
$G:=((g_{ij}))$, where
$$g_{ij}=-\alpha D(Q_i)\cdot D(Q_j)-(K_{X_{\C}}\cdot D(Q_i))(K_{X_{\C}}\cdot D(Q_j)),$$
$\alpha$ being 9 (resp. 8) if $X=\bP^2$ (resp. $X$ is a Hirzebruch
surface). The configuration $\C$ is said to be {\it P-sufficient} if
$xGx^t>0$ for all vectors $(x_0,\ldots,x_n)\in
\bR^{n+1}\setminus\{0\}$ such that $x_i\geq 0$ for all $i$.
 }
\end{defn}
\begin{rem}
{\rm Recall that the coefficients of each divisor $D(Q_i)$ are those
appearing in the $i$th row of the inverse of the proximity matrix of
$\C$. Hence, the matrix $G$ is easy to compute and depends only on
the proximity relations among the points in $\C$. }
\end{rem}
A general method to decide if a configuration is P-sufficient or not
is given in \cite{gaddum}, which consists of checking the
non-emptiness of certain sets defined by linear inequalities. Also,
a configuration is P-sufficient whenever all the entries of the
matrix $G$ are non-negative and the diagonal ones are strictly
positive (for an example, see \cite[page 86]{galmon}). Furthermore,
when the configuration is a chain, it is very easy to decide if it
is P-sufficient or not:
\begin{prop}\cite[Cor. 2]{galmon3}
When $\C$ is a chain constellation, $\C$ is P-sufficient if and only
if the last entry of the matrix $G$ is strictly positive.
\end{prop}
The following result is proved in \cite[Th. 2]{galmon} when
$X=\bP^2$ and in \cite[Th. 1]{galmon3} in the general case, and it
shows that the P-sufficient configurations give raise to surfaces
with polyhedral cones of curves.
\begin{thm}
If $\C$ is a P-sufficient configuration then the cone of curves
$NE(Z)$ is polyhedral.
\end{thm}
\begin{rem}
{\rm It can be proved that, if $X=\bP^2$ (resp. $X$ is a Hirzebruch
surface) and the cardinality of $\C$ is $\leq 8$ (resp. $\leq 7$)
then $\C$ is P-sufficient. For an example of a P-sufficient
configuration with 11 points see \cite[page 86]{galmon}.
 }
\end{rem}
\begin{rem}
{\rm If $NE(Z)$ is polyhedral, the configuration $\C$ may not be
P-sufficient. For example, take a configuration consisting of 9 or
more proper points on a conic. The cone $NE(Z)$ is polyhedral (by
Proposition \ref{conic}) but, however, the configuration is not
P-sufficient. }
\end{rem}
\begin{rem}
{\rm As a result which follows from \cite{K}, the topological cells
of the characteristic cone $\tilde{P}(Z)$ (see \cite[page 340]{K}
for the definition) correspond one to one to surjective morphisms
from $Z$ to a (connected) normal variety ({\it contractions}). Since
there is an injection between the set of topological cells of
$\tilde{P}(Z)$ and the one of $P(Z)$ \cite[Th. 2.1]{CG}, one has
that the polyhedrality of $NE(Z)$ implies that the number of
contractions is finite. }
\end{rem}
\subsection{Clusters of base points associated with pencils on surfaces}\label{sect33}
We shall consider now a particular type of linear systems on a
projective regular surface $X$: given an effective divisor $H$,
$\varrho$ will be a linear sub-system of $|H|$ without fixed
components and with projective dimension 1 (a {\it pencil} in the
sequel). Such a pencil $\varrho$ corresponds to the projectivization
of the sub-vector space $V_{\varrho}$ of $H^0(X,\OO_{X}(H))$ given
by $\{s \in H^0(X,\OO_X(H)) \mid (s)_0\in \varrho\}\cup \{0\},$
$(s)_0$ denoting the divisor of zeros of the section $s$. If
$\I(\varrho)$ is the ideal sheaf on $X$ defining the base point
scheme of $\varrho$, consider the associated cluster of base points
$\K_{\I(\varrho)}=(\C_{\I(\varrho)},\underline{m})$ and the
associated divisor $D(\K_{\I(\varrho)})$.

\vsd \noindent {\em Cones of curves and irreducible pencils} \vsd

One can consider the linear system on $X$ of all effective divisors
in $|H|$ passing through the cluster $\K_{\I(\varrho)}$, which will
be denoted by ${\mathcal L}_H(\K_{\I(\varrho)})$. Then, it is clear
that $\varrho \subseteq {\mathcal L}_H(\K_{\I(\varrho)})$. The
question we propose to answer now is the following one: when is this
inclusion an equality? or, equivalently, when is a pencil determined
by the class of $H$ in the Picard group and its cluster of base
points? A fixed basis of $V_{\varrho}$ provides a rational map $f: X
\cdots \rightarrow \bP V_{\varrho}\cong \bP^1$ (actually this map is
independent from the basis up to composition with an automorphism of
$\bP^1$). The closures of the fibers of $f$ are exactly the curves
of the pencil $\varrho$. For this reason, the elements of $\varrho$
are usually called {\it fibers}. Moreover, the morphism
$\sigma:=\sigma_{\C_{\I(\varrho)}}:Z:=X_{\C_{\I(\varrho)}}\rightarrow
X$ is defined by the virtual transform on $Z$ of the chosen basis of
$V_{\varrho}$ with respect to the cluster of base points
$\K_{\I(\varrho)}$, and it is a minimal composition of point
blowing-ups eliminating the indeterminacies of the rational map $f$,
that is, the map $h:= f \circ \sigma:Z\rightarrow \bP^1$ is a
morphism (see \cite[Th. II.7]{Beau}).
It is clear that, if $C_1$ and $C_2$ are two curves on $Z$ such that
$C_1+C_2$ is contracted (to a closed point) by $h$, then both curves
$C_1, C_2$ must also be contracted by $h$. Therefore, the images in
$A(Z)$ of all curves of $Z$ which are contracted by $h$ generate a
face of the cone of curves $NE(Z)$, that we shall denote by
$\Delta_{\varrho}$.
Consider the divisor $G_{\varrho}:=\sigma^*H-D(\K_{\I(\varrho)})$.
The linear system $\check{\varrho}$ on $Z$ given by the virtual
transforms of the curves in $\varrho$ with respect to the cluster of
base points is contained in the complete linear system
$|G_{\varrho}|$, and the image of $G_{\varrho}$ in $A(Z)$ is the
same than the one of a general fiber of the pencil $\varrho$. As a
consequence, $G_{\varrho}$ is a nef divisor (since its associated
complete linear system is base point free). Therefore, we conclude
that $\Theta_{\varrho}:=NE(Z)\cap G_{\varrho}^{\perp}$ is a face of
$NE(Z)$.

\begin{lem}\label{faces}
Both faces $\Delta_{\varrho}$ and $\Theta_{\varrho}$ coincide.
\end{lem}
\begin{proof}
The morphism $h$ factorizes as $h=t\circ \phi$, where $\phi
:Z\rightarrow \bP H^0(Z,\OO_{Z}(G_{\varrho}))$ is the morphism
induced by a basis of $H^0(Z,\OO_{Z}(G_{\varrho}))$ obtained by
completing the one given by the virtual transform of the fixed basis
of $V_{\varrho}$, and $t: \bP H^0(Z,\OO_{Z}(G_{\varrho}))- -
\rightarrow \bP V_{\varrho}\cong \bP^1$ is the projection. A curve
$C$ is contracted by $\phi$ if and only if $G_{\varrho}\cdot C=0$,
and it is obvious that, in this case, it is also contracted by $h$.
Therefore, one has that $\Theta_{\varrho}\subseteq
\Delta_{\varrho}$. Since the strict transforms on $Z$ of two general
fibers do not meet, one has that $G_{\varrho}^2=0$. When $\dim
A(Z)\geq 3$, the hyperplane $G_{\varrho}^{\perp}$ is tangent to the
cone $Q(Z)$ defined in the preceding section. So, we have the
following equivalence (which is also valid when $\dim A(Z)=2$):
$$  x\in G_{\varrho}^{\perp}\setminus \{0\} \mbox{ and } x^2<0 \mbox{ if and only if  $x$ is not a (real) multiple of $G_{\varrho}$}.$$
It is clear that there exists $y\in P(Z)\setminus \{0\}$ such that
$\Delta_{\varrho}\subseteq NE(Z)\cap y^{\perp}$. But $y$ belongs to
$G_{\varrho}^{\perp}$, since $G_{\varrho}\in \Delta_{\varrho}$. So,
by the above equivalence, $y$ is a multiple of $G_{\varrho}$ and
hence $\Delta_{\varrho}\subseteq \Theta_{\varrho}$.
\end{proof}
Notice that the  integral curves which are contracted by $h$ are
exactly the strict transforms of the integral components of the
fibers of the pencil $\varrho$ and some strict transforms of
exceptional divisors (the so-called {\it vertical} exceptional
divisors). From this consideration and the above ones it follows the
next result, whose proof is also implicit in \cite{galmon2} but in a
different framework.
\begin{prop}
An integral curve $C$ on $X$ is a component of a fiber of the pencil
$\varrho$ if and only if its strict transform $\tilde{C}$ on $Z$
satisfies that $G_{\varrho}\cdot \tilde{C}=0$. Moreover, in this
case, $\tilde{C}^2\leq 0$.
\end{prop}

We fix a closed immersion $i: X \hookrightarrow \mathbb{P}^s$ of $X$
into a projective space; the {\it degree} of a curve $F$ on $X$ will
be the intersection product $i^*{\mathcal O}_{\mathbb{P}^s}(1)\cdot
F$.

The pencil $\varrho$ is said to be {\it irreducible} if it has
integral general fibers. The following proposition is also proved in
\cite{galmon2} when $X$ is the projective plane using the Cayley-Bacharach Theorem, but we show here
a different proof.
\begin{prop}\label{irred}
If $\varrho$ is irreducible then $\varrho ={\mathcal
L}_{H}(\K_{\I(\varrho)})$, that is, it is determined by the class of $H$ in the Picard group and
its cluster of base points.
\end{prop}
\begin{proof}
Reasoning by contradiction, assume that the projective dimension of
${\mathcal L}_{H}(\K_{\I(\varrho)})$ is greater than 1. Take a curve
$C\in {\mathcal L}_{H}(\K_{\I(\varrho)})$ such that $C\not\in
\varrho$. The image in $A(Z)$ of the virtual transform of $C$ on $Z$
with respect to the cluster $\K_{\I(\varrho)}$ belongs to
$|G_{\varrho}|$ and therefore, if $\tilde{C}$ denotes the strict
transform, it holds that $G_{\varrho}\cdot \tilde{C}=0$ (taking into
account that $G_{\varrho}$ is nef and $G_{\varrho}^2=0$). Then,
applying the above result, one has that the integral components of
$C$ are fibers of the pencil $\varrho$. If $C_1$ denotes one of
these components, one has that its degree is strictly less than the
one of the general fibers of $\varrho$ (otherwise, on the one hand,
$C_1$ would be a fiber of the pencil and, on the other hand, it
would coincide with $C$,  a contradiction). But, since the pencil is
irreducible, the number of reducible fibers is finite. The
contradiction follows from the fact that, by the initial assumption,
there are infinitely many curves $C$ as above.
\end{proof}
Assuming that the characteristic of the ground field is 0 we prove
the converse of the above statement:
\begin{prop}
If ${\rm char}(\bK)=0$ then $\varrho$ is an irreducible pencil if
and only if $\varrho ={\mathcal L}_{H}(\K_{\I(\varrho)})$.
\end{prop}
\begin{proof}
Reasoning by contradiction, assume that $\varrho ={\mathcal
L}_d(\K_{\I(\varrho)})$ and $\varrho$ is not irreducible. Then $\varrho$ is {\it composite with an
irreducible pencil}, that is, there exist rational maps
$q_1:X\cdots \rightarrow \bP^1$ and $q_2:\bP^1\cdots \rightarrow
\bP^1$ such that the closures of the fibers of $q_1$ correspond to
an irreducible pencil $\varsigma$ of degree $e$, $q_2$ is
generically finite of degree $n:=d/e>1$ and $f=q_2\circ q_1$.
Therefore, if $C$ is a general fiber of $\varsigma$, taking into
account that $n\tilde{C}$ is linearly equivalent to $G_{\varrho}$,
one gets that $nC$ is a fiber of $\varrho$. Hence, general fibers of
$\varrho$ are not reduced, which is a contradiction.
\end{proof}

\vsd \noindent {\em Pencils at infinity}\vsd

We shall assume until the end of this subsection that ${\rm
char}(\bK)=0$. A specially interesting class of pencils on the
projective plane $\bP^2$ are the so-called {\it pencils at
infinity}.
\begin{defn}
{\rm Taking homogeneous
coordinates $(X_1:X_2:X_3)$ on $\bP^2$, a (linear) pencil (without fixed
components) $\varrho \subseteq |\OO_{\bP^2}(d)|$, $d\in \bZ_+$, is said to be {\it at
infinity} if $V_{\varrho}=\langle F,X_3^d \rangle$, where $F(X_1,X_2,X_3)$
is an homogeneous polynomial of degree $d$ and $X_3=0$ is considered
as the line of infinity. }
\end{defn}
A particular case of pencil at infinity is obtained when $F=0$
defines a curve $C$ {\it having one place at infinity}, that is, it
intersects with the line of infinity only in a single point $Q$ and
$C$ is reduced and unibranched at $Q$. This is easily seen to imply
that $C$ is integral. This type of curves have been extensively
studied by several authors as Abhyankar, Moh, Satayhe and Suzuki
\cite{abh, abh2, abh1, moh, sat, suzuki}. All the curves in the
pencil at infinity defined by $F$, except the non-reduced one, have
one place at infinity and their singularities at the point of
infinity have the same minimal embedded resolution than the one of
$C$ \cite{moh}. In  \cite{CPR} is proved a structure theorem for the
cone of curves and the characteristic cone of the surface $Z$
obtained by blowing-up the configuration of base points of a pencil
of this type (which is a chain constellation). Actually, instead of
the cone of curves, the {\it effective semigroup} is considered; it
is the sub-semigroup $NE_S(Z)$ of ${\rm Pic}(Z)$ spanned by the
classes of the effective divisors.
\begin{thm}\cite{CPR}\label{thm32}
Let $\varrho$ be a pencil at infinity such that $V_{\varrho}=\langle F,X_3^d \rangle$, where $F=0$ defines a curve having one place at infinity of degree $d\geq 1$, and let $Z$ be the surface obtained by blowing-up the points in the configuration ${\mathcal C}_{\mathcal{I}(\varrho)}$. Then:
\begin{itemize}
\item[(a)] The semigroup $NE_S(Z)$ is spanned by the strict transform of the line of infinity and the strict transforms of the exceptional divisors.
\item[(b)] The cones $P(Z)$ and $\tilde{P}(Z)$ coincide and are regular.
\end{itemize}
\end{thm}
Recall that, in the local case, Enriques solved the problem of determining when there exists a germ of curve passing effectively through a cluster $(\C, \underline{m})$ leading to the proximity inequalities \cite{EC,Ca}, and Zariski considered the semigroup of $\sigma_{\C}$-generated line bundles in order to establish the unique factorization of complete ideals \cite{ZS}. In our case, global analogues to these problems consist of characterizing the semigroup $P_S^{st}(Z)\subseteq {\rm Pic}(Z)$ generated by the classes  of the strict transforms on $Z$ of curves on $\bP^2$, and the semigroup $\tilde{P}_S(Z)\subseteq {\rm Pic}(Z)$  generated by the classes of divisors $D$ on $Z$ such that $\OO_{\bP^2}(D)$ is generated by global sections (notice that one has $\tilde{P}_S(Z)\subseteq P_S^{st}(Z)\subseteq P_S(Z)$, where $P_S(Z)$ denotes the semigroup generated by the nef classes). Concerning these questions, in \cite{CPR} it is proved the following result:
\begin{thm}\label{thm33}
Let $\varrho$ and $Z$ be as in Theorem \ref{thm32}.
\begin{itemize}
\item[(a)] A divisor class $D$ belongs to $P_S^{st}(Z)$ if and only if $\OO_{\bP^2}(D)$ is generated by global sections except possibly at finitely many closed points.
\item[(b)] $\tilde{P}_S(Z)= P_S^{st}(Z)= P_S(Z)$ if and only if all curves in the pencil $\varrho$ (except the non-reduced one) are rational.
\end{itemize}
\end{thm}
Due to Theorem \ref{thm32} the cone of curves $NE(Z)$ associated to a pencil defined by a curve having one place at infinity is always polyhedral. Since P-sufficient configurations give rise to polyhedral cones of curves, a natural question arises: when is the configuration of base points of such a pencil P-sufficient? The answer is given in the following proposition:
\begin{prop}\label{prop36} \cite[Prop. 3]{mons} 
Let $\varrho$ and $Z$ be as in Theorem \ref{thm32}. The
configuration ${\mathcal C}_{\I(\varrho)}$ is P-sufficient if and
only if $F=0$ defines an Abhyankar-Moh-Suzuki curve (i.e. it is
rational and smooth in its affine part).
\end{prop}
Consider now a pencil at infinity $\varrho$ given by an homogeneous polynomial $F$ of degree $d\geq 1$ but not necessarily defining a curve having one place at infinity. Assuming certain conditions on $F$ it is possible to describe the structure of the effective semigroup and the characteristic cone of the surface $Z$ obtained by blowing-up the configuration of base points of the pencil:
\begin{thm}\cite[Th. 3]{CPR2}
Assume that  $\varrho$ is an irreducible pencil at infinity such that $F$ factorizes as $F^{a_1}_1\cdots F^{a_s}_s$, where $d_1,\ldots,d_s,a_1,\ldots,a_s \in \bZ_+$, $\gcd(a_1,\ldots,a_s)=1$, $F_1,\dots,F_s$ are homogeneous polynomials of respective degrees $d_1,\ldots,d_s$ such that the curves defined by $F_i=0$ have one place at infinity and, if $s\geq 2$, $F_i\not\in \langle F_1, Z^{d_1} \rangle$ for some $i$, $2\leq i\leq s$. Then,
\begin{itemize}
\item[(a)] The effective semigroup $NE_S(Z)$ is spanned by the strict transforms on $Z$ of the following curves: the exceptional divisors, the line of infinity and the curves defined by the polynomials $F_i$, $1\leq i\leq s$.
\item[(b)] $\tilde{P}(Z)=P(Z)$.
\end{itemize}
\end{thm}
\section{Applications}\label{aplications}
\subsection{Applications to the Poincar\'e Problem}\label{foliations}
Some progress concerning the theory of foliations have been done by
using, as a tool, the language of configurations and clusters and,
also, considerations involving cones of curves  (sections
\ref{sect32} and \ref{sect33}). They are related to the so-called
Poincar\'e problem. We shall summarize some of such progress but,
previously, we will introduce briefly some background on the theory
of foliations.

\vsd \noindent {\em Background}\vsd

 We shall assume that $\bK=\gc$.
An (algebraic singular) {\it foliation} $\cf$ on a projective smooth
surface (a surface in the sequel) $X$ can be defined by a collection
$\{(U_i,\omega_i)\}_{i\in I}$, where $\{U_i\}_{i\in I}$ is an open
covering of $X$, $\omega_i$ is a non-zero regular differential
1-form on $U_i$ with isolated zeros and, for each couple $(i,j)\in
I\times I,$
\begin{equation}\label{eq}
\omega_i=g_{ij}w_j \mbox{\;\;\;on \;\;} U_i \cap U_j,
\mbox{\;\;\;\;} g_{ij}\in \co_X(U_i\cap U_j)^*.
\end{equation}
The singular locus $Sing(\cf)$ of $\cf$ is the discrete subset of $X$ defined by $$Sing(\cf)\cap U_i=\mbox{ zeroes of }w_i.$$
The transition functions $g_{ij}$ of a foliation $\cf$ define an
invertible sheaf ${\mathcal L}$ on $X$
and the relations (\ref{eq}) can be thought as defining relations
of a global section of the sheaf ${\mathcal L}\otimes \Omega_X^1$,
which has isolated zeros (because each $\omega_i$ has isolated
zeros). This section is uniquely determined by the foliation
$\cf$, up to multiplication by a non zero element in ${\mathbb C}$.
Conversely, given an invertible sheaf ${\mathcal L}$ on $X$, any
global section of ${\mathcal L}\otimes \Omega_X^1$ with isolated zeros
defines a foliation $\cf$.
Alternatively, a foliation can also be defined by a collection $\{(U_i,v_i)\}_{i\in I}$, where $v_i$ is a vector field on $U_i$ with isolated zeroes, satisfying analogous relations as in (\ref{eq}).
Given $P\in X$, a {\it (formal) solution of $\cf$ at $P$} will be
an irreducible element $f\in \widehat{\co}_{X,P}$ (where
$\widehat{\co}_{X,P}$ is the ${\rm m}_P$-adic completion of the
local ring $\co_{X,P}$ and ${\rm m}_P$ its maximal ideal) such
that the local differential 2-form $\omega_P \wedge df$ is a
multiple of $f$, $w_P$ being a local equation of $\cf$ at $P$. An
element in $\widehat{\co}_{X,P}$ will be said to be {\it invariant
by} $\cf$ if all its irreducible components are solutions of $\cf$
at $P$. An {\it algebraic solution} of $\cf$ will be an integral
(i.e. reduced and irreducible) curve $C$ on $X$ such that its
local equation at each point in its support is invariant by $\cf$.
Moreover, if every integral component of a curve $D$ on $X$ is an
algebraic solution, we shall say that $D$ {\it is invariant by}
$\cf$.
Seidenberg's result of
reduction of singularities \cite{seid} proves  that there is a
sequence of blowing-ups
\begin{equation}\label{seq}
X_{n+1} \mathop  {\longrightarrow} \limits^{\pi _{n} }
X_{n} \mathop {\longrightarrow} \limits^{\pi _{n-1} }  \cdots
\mathop {\longrightarrow} \limits^{\pi _2 } X_2 \mathop
{\longrightarrow} \limits^{\pi _1 } X_1 : = X
\end{equation}
such that the strict
transform $\cf_{n+1}$ of $\cf$ on the last obtained surface
$X_{n+1}$ has only certain type of singularities which cannot be
removed by blowing-up, called {\it simple} singularities. Such a
sequence of blowing-ups is called a {\it resolution} of $\cf$, and it
will be {\it minimal} if it is so with respect to the number of
involved blowing-ups. Assuming that the above sequence of blowing-ups is a minimal
resolution of $\cf$, we shall denote by $\C_{\cf}$ the associated
configuration $\{P_i\}_{i=1}^n$ given by the centers of the blowing-ups. Note that each point $P_i$ is an
ordinary (that is, not simple) singularity  of the foliation
$\cf_i$.
An exceptional divisor $B_{P_i}$ (respectively, a point $P_i\in
\C_{\cf}$)  is
called {\it non-dicritical} if it is invariant by the foliation
$\cf_{i+1}$ (respectively, all the exceptional divisors $B_{P_j}$,
with $P_j\geq P_i$, are non-dicritical). Otherwise, $B_{P_i}$
(respectively,  $P_i$) is said to be {\it dicritical}.
Particularizing to the projective plane, it holds that
a foliation $\cf$ on $\gp^2$ (of degree $r$) can be defined by means of a projective 1-form
$$
\Omega=A dX_1 + B dX_2 + C dX_3,
$$
where $A,B$ and $C$ are homogeneous polynomials of degree $r+1$
without common factors which satisfy the Euler's condition
$X_1A+X_2B+X_3C=0$ (see \cite{G-M}). From a more geometrical point
of view, $\cf$ can be regarded as the rational map $\Phi:\gp^2\cdots
\rightarrow \check{\gp}^2$ which sends a point $P$ to
$(A(P):B(P):C(P))$. The singular locus of $\cf$ is the set of points
where this rational map is not defined, that is,  the set of common
zeros of the polynomials $A, B$ and $C$. Moreover, a curve $D$ on
$\gp^2$ is invariant by $\cf$ if, and only if, $G$ divides the
projective 2-form $dG \wedge \Omega$, where $G(X_1,X_2,X_3)=0$ is an
equation of $D$.

\vsd \noindent {\em The Poincar\'e problem}\vsd

 We begin with a
differential equation with polynomial coefficients of order 1 and
degree 1, that is, of the type $Q(x,y)y'+P(x,y)=0$, with $P,Q\in
\gc[x,y]$ or, in a more general form, given by the vector field
$D=Q(x,y)\partial / \partial x-P(x,y) \partial/
\partial y$ or, equivalently, the differential form $\omega=P(x,y)dx
+ Q(x,y) dy$. It is said that the differential equation is {\it
algebraically integrable} or that it has a {\it rational first
integral} if there exists a rational function $R=\frac{f}{g}$,
$f,g\in \gc[x,y]$, such that $\omega \wedge dR=0$ (or equivalently
$D(R)=0$). This implies that the function $R$ is constant on the
solutions of the equation, that is, these are the curves whose
implicit equations are of the form $\lambda f + \mu g=0$, $\lambda,
\mu \in \gc$ (hence all solutions are algebraic curves). In 1891, H.
Poincar\'e \cite{Poi1} observed that, once we possess a bound on the
degree of a polynomial defining a general irreducible solution, we
can try  to find the rational first integral by making purely
algebraic computations. The problem of finding this bound in terms
of the degree of the foliation is classically known as the
Poincar\'e problem, although it was studied before by Darboux and
also by Painlev\'e and Autonne more or less at the same time than
Poincar\'e. From a more modern point of view, a vector field on the
affine plane is given by polynomial coefficients if and only if it
is the restriction of a foliation of the projective plane. So, the
Poincar\'e problem can be treated in this framework. Then, we shall
say that a foliation $\cf$ of $\bP^2$ has a {\it rational first
integral} if there exists a rational function $R$ of $\gp^2$ such
that $dR \wedge \Omega=0$. The substantial current interest in the
Poincar\'e problem was stimulated by Cerveau and Lins Neto in
\cite{ce-li}. In this paper, the problem is stated in a more general
form, avoiding the assumption of the algebraic integrability. That
is, if we assume that a foliation $\cf$ of $\bP^2$ has an algebraic
solution $C$, can we give conditions that allow us to bound $\deg C$
in terms of $\deg \cf$? The main result of \cite{ce-li} gives an
answer assuming that all the singularities of $C$ are simple nodes
(in this case $\deg C\leq \deg \cf+2$). Carnicer, in \cite{Car},
proves the same inequality in the case that $C$ does not pass
through dicritical singularities of $\cf$. However, there exist
examples showing that, in general, $\deg C$ cannot be bounded in
terms of $\deg \cf$. A remarkable counterexample is given in
\cite{l-n} (families of algebraically integrable foliations of fixed
degree and singularities of fixed analytic type are given, in such a
way that the general algebraic solutions have arbitrarily big
degree).

\vsd \noindent {\em Results using infinitely near points}\vsd

In \cite{Ca-Ca}, Carnicer and the first author extend in certain
manner the result given in \cite{Car} when dicritical singularities
appear. They use, as an important tool and unifying element in the
paper, the language of infinitely near points and proximity. In
fact, they prove proximity formulae for foliations \cite[Prop.
3.5]{Ca-Ca} and use them to give relations between local invariants
of an algebraic solution and local invariants of the foliation. To
state the main result, we need to introduce some notations.

Given a reduced invariant curve $C$ of a foliation $\cf$ of a
projective smooth surface $X$, let $\N_C$ be the configuration over
$X$ that consists of those points in $\C_{\cf}$ whose image on $X$
by the composition of blowing-ups given in (\ref{seq}) belongs to
$C$. For each $P\in \N_C$ denote by $s_P(\cf)$ the number of points
$Q\in \N_C$ such that $P\rightarrow Q$ and the exceptional divisor
$B_Q$ is non-dicritical, and set $\nu_P(C)$ (resp. $\nu_{P}(\cf)$)
the multiplicity at $P$ of the strict transform of $C$ (resp. the
minimum order of the coefficients of a local differential form
defining the strict transform of $\cf$ at $P$).
\begin{thm}\cite[Th. 1]{Ca-Ca}\label{Ca-Ca}
Let $\cf$ be a foliation of $\gp^2$ and $C$ a reduced curve which is invariant
by $\cf$. Let $\K=(\N_C=\{Q_i\}_{i=1}^t,\underline{m})$ be the cluster
such that $m_i:=\nu_{Q_i}(C)+ s_{Q_i}(\cf)-\nu_{Q_i}(\cf)-1$ for all $i\in \{1,\ldots,t\}$. Let $d$ be a non-negative integer such that the linear system
${\mathcal L}_d(\K)$ (of curves of degree $d$ passing through $\K$) is not empty. Then: $$ \deg (C)\leq \deg(\cf)+2+d.$$
\end{thm}
As an application, if either we fix the number of tangents at the
singular points of $\cf$, or we fix the equisingularity types of the
curve at the singular points of $\cf$, then concrete values of $d$
can be obtained from the fixed data computing the linear system
${\mathcal L}_d(\K)$ (this involves the resolution of a system of
linear equations). The obtained bounds will be valid for particular
types of invariant curves. It is worth adding that the above
mentioned results of Cerveau and Lins Neto, and Carnicer are
particular cases of Theorem \ref{Ca-Ca} since, in both cases, it can
be proved that $\nu_{Q_i}(\cf)+1\geq \nu_{Q_i}(C)+ s_{Q_i}(\cf)$ for
all $Q_i\in \N_C$ (then, the results follow by taking $d=0$). In
\cite{C-G-GF-R}, the result given in Theorem \ref{Ca-Ca} is
generalized for foliations of arbitrary projective smooth surfaces.
In this case, one looses the concept of degree and, in addition to
configurations and proximity, the use of divisors and Intersection
Theory is required. The main result is the following one:
\begin{thm}\cite[Th. 2]{C-G-GF-R}\label{ggg}
Let $X$ be a projective smooth algebraic surface, $\cf$ a foliation of $X$,
$C$ a reduced curve which is invariant by $\cf$ and $\K=(\N_C,\underline{m})$ as in
the statement of Theorem \ref{Ca-Ca}. If  $H$ is a divisor such
that the linear system ${\mathcal L}_{H}(\K)$ is not empty then:
$$(D_{\cf}+H-C)\cdot C_1\geq 0,$$
where $D_{\cf}$ is a Cartier divisor in the divisor class defined by the transition functions $g_{ij}$
associated with the foliation (see the beginning of the section) and $C_1$ is the
reduced curve consisting of the components of $C$ not contained in the support of $H$.
\end{thm}
\begin{rem}
{\rm
If we take, in Theorem \ref{ggg}, $X=\gp^2$, then one has that $D_{\cf}=(\deg(\cf)+2)L$,
where $L$ is a line. Setting $d=\deg(H)$ one obtains the inequality $(\deg(\cf)+2+d-\deg(C)) \deg(C)\geq 0$,
that is, $\deg(C)\leq \deg(\cf)+2+d$. Hence, one recovers Theorem \ref{Ca-Ca}.
}
\end{rem}
Also, a generalization of Theorem \ref{Ca-Ca} for 1-dimensional foliations on the
projective space $\gp^n$ is given in \cite{Ca-Ca-GF}, also expressed in terms of
infinitely near points.
The above results have been improved by Esteves and Kleiman in
\cite{K-E} using arguments which do not involve infinitely near
points.
A recent result concerning the Poincar\'e problem is given by C.
Galindo and the third author in \cite{galmon2}. In it, it is
provided an algorithm to decide whether a foliation of $\gp^2$ has a
rational first integral and to compute it in the affirmative case.
This algorithm runs whenever we assume the polyhedrality of the cone
of curves of the surface obtained by blowing-up the configuration
dicritical points in $\C_{\cf}$, which we shall denote by $\B_{\cf}$
(this happens, for instance, when this configuration is
P-sufficient). The inputs of the algorithm are the projective
differential 1-form $\Omega$ defining the foliation, the
configuration $\B_{\cf}$ of dicritical points and the non-dicritical
exceptional divisors coming from $\B_{\cf}$. We shall explain now
the main ideas that give rise to that result.
Assume now that $\cf$ is a foliation of $\gp^2$ and let
$\pi_{\cf}:Z_{\cf}\rightarrow \gp^2$ be the composition of
blowing-ups of the configuration $\B_{\cf}$. We shall also assume
that the cardinality of $\B_{\cf}$ is greater than 1.
If $\cf$ has a rational first integral one has the following
fundamental facts:
\begin{itemize}
\item[(1)] A rational first integral $R$ can be taken to be the
quotient of two homogeneous polynomial of the same degree $d$, $F$
and $G$, such that the pencil $\varrho\subseteq |\OO_{\gp^2}(d)|$
that they provide is irreducible.
\item[(2)] $\B_{\cf}$ coincides with the
configuration of base points $\C_{\I(\varrho)}$ of the pencil
$\varrho$ \cite[Prop. 1]{galmon2}.
\item[(3)] The algebraic solutions of $\cf$ are the integral
components of the curves in the pencil $\varrho$ and the images of
their strict transforms on $A(Z_{\cf})$, together with the strict
transforms of the {\it vertical} exceptional divisors, generates the
face of $NE(Z_{\cf})$ given by $G_{\varrho}^{\perp}\cap NE(Z_{\cf})$
see Section \ref{sect33}, after Lemma \ref{faces}). Moreover, the
{\it vertical} exceptional divisors are exactly the {\it
non-dicritical} exceptional divisors (as a consequence of
\cite[Prop. 2.5.2.1]{julio} and \cite[Exercise 7.2]{casas}).
\end{itemize}
Set $n$ the cardinality of the configuration $\B_{\cf}$ (then, $\dim
A(Z_{\cf})=n+1$).
\begin{defn}
{\rm An {\it independent system of algebraic solutions} for $\cf$
will be a set $S=\{C_1,\ldots,C_s\}$ such that  the system
$$\A_S:=\{\tilde{C}_1,\ldots, \tilde{C_s},E_{i_1},\ldots, E_{i_{n-s}}\}\subseteq A(Z_{\cf})$$ is
$\gr$- linearly independent, where $\tilde{C}_i$ denotes the strict
transform of $C_i$ on $Z_{\cf}$ and $\{E_{i_k}\}_{k=1}^{n-s}$ are
the strict transforms of the non-dicritical exceptional divisors
corresponding to points in $\B_{\cf}$. }
\end{defn}
If the cone of curves $NE(Z_{\cf})$ is polyhedral and $\cf$ has a
rational first integral one has that the face
$G_{\varrho}^{\perp}\cap NE(Z_{\cf})$ has codimension 1 and,
therefore, it is spanned by $n$ $\gr$-linearly independent
generators of extremal rays. Moreover, due to the polyhedrality of
$NE(Z_{\cf})$ and the inclusion $Q(Z_{\cf})\subseteq NE(Z_{\cf})$,
one has that a ray in $A(Z_{\cf})$ is an extremal ray of
$NE(Z_{\cf})$ if and only if it is spanned by the image in
$A(Z_{\cf})$ of an integral curve on $Z_{\cf}$ with strictly
negative self-intersection. Therefore, we can conclude the following
\begin{prop}
If $NE(Z_{\cf})$ is polyhedral and $\cf$ has a rational first
integral then there exists an independent system of algebraic
solutions $S$ such that $\tilde{C}^2<0$ for all $C\in S$. Moreover,
the hyperplane $G_{\varrho}^{\perp}$ is generated by $\A_S$.
\end{prop}
We assume from now on that $NE(Z_{\cf})$ is a polyhedral cone. The
above mentioned algorithm consists of two parts. In the first one,
from the data $\{\Omega, \B_{\cf}, (E_{i_1},\ldots,E_{i_{n-s}})\}$
(which comes from the resolution of the singularities of $\cf$),
either one concludes that $\cf$ has no rational first integral, or
an independent system of algebraic solutions is returned \cite[Alg.
3]{galmon2}. The algorithm generates a strictly increasing sequence
of convex cones $V_0 \subset V_1\subset \cdots$ such that $V_0$ is
generated by $\{{E}_Q\}_{Q\in {\mathcal B}_{\cf}}$ and $V_i$ is
generated by $V_{i-1} \cup \{\tilde{Q}_i\}$ for $i\geq 1$, where $
Q_1, Q_2,\dots$ are curves on $\gp^2$ (ordered with non-decreasing
degrees) satisfying certain conditions, being $\tilde{Q}_i^2<0$
among them. We stop when one of the following cases occurs: (1)
there exists a subset $S$ of $\{Q_1,\ldots, Q_i\}$ which is an
independent system of algebraic solutions, or (2)
$Q(Z_{\cf})\subseteq V_i$. If case (2) holds but (1) does not occur,
then we conclude that {\it $\cf$ has no rational first integral}
(see the explanation of Algorithm 3 of \cite{galmon2}). The second
part of the algorithm \cite[Alg. 2]{galmon2} will be applied  when,
in the first part, an independent system of algebraic solutions $S$
has been obtained. In the case that $\cf$ had rational first
integral, the hyperplane $G_{\varrho}^{\perp}$ would be the one
generated by $\A_S$. Hence, the first step will be to compute the
primitive (in the lattice $\gz^{n+1}\cong {\rm
Pic}(Z_{\cf})\subseteq A(Z_{\cf})$) class $T_{\cf,S}$ such that the
hyperplane $T_{\cf,S}^{\perp}$ is the one generated by $\A_S$ and
$T_{\cf,S}\cdot \pi_{\cf}^*L>0$ for a line $L$ (see \cite[page
8]{galmon2}). Notice that if $\cf$ had a rational first integral
then $G_{\varrho}$ should be equal to $\alpha T_{\cf,S}$ for some
integer $\alpha>0$. By \cite[Prop. 4]{galmon2} one of the following
conditions is satisfied:
\begin{itemize}
\item[(1)] $T_{\cf,S}^2\not=0$. In this case {\it $\cf$ has not a
rational first integral} (since, otherwise, the equality
$G_{\varrho}^2=0$ gives a contradiction).
\item[(2)] The coefficients of all the elements of $\A_S$ in the
decomposition of $T_{\cf,S}$ as linear combination of $\A_S$ are
strictly positive. In this case \cite[Th. 2]{galmon2} shows the
existence of a unique possible value for $\alpha$. If the dimension
of the space of global sections $H^0(\gp^2,
{\pi_{\cf}}_*\OO_{Z_{\cf}}(\alpha T_{\cf,S}))$ is not 2, then {\it
$\cf$ has not a rational first integral} (otherwise we have a
contradiction, since $V_{\varrho}$ would coincide with this space by
Prop. \ref{irred}). If the above dimension is 2, then one can
compute a basis $\{F,G\}$ of the space. If $\Omega \wedge d(F/G)=0$
then {\it $F/G$ is a first integral}; otherwise, {\it $\cf$ is not
algebraically integrable}.
\item[(3)] The set $\{\lambda\in \bZ_+
\mid h^0(\gp^2,{\pi_{\cf}}_*\co_{Z_{\cf}}(\lambda T_{{\mathcal
F},S}))\geq 2\}$ is not empty. In this case one can take $\alpha$ to
be the minimum of this set and proceed as in the above case.
\end{itemize}
\begin{rem}
{\rm Although the algorithm is expressed, for clarity, in terms of
divisors, it involves the computation of linear systems of plane
curves coming from clusters. For instance, to find the curves
$Q_1,Q_2,\ldots$ one takes clusters $\K=(\B_{\cf},\underline{m})$
and, beginning with $d=1$ and increasing $d$ successively, computes
(for each fixed value of $d$) all the linear systems ${\mathcal
L}_d(\K)$ with $(d; \underline{m})$ satisfying certain properties:
$d^2-\sum m_i^2<0$, the proximity inequalities and other properties
coming from the adjunction formula. These properties come from the
fact that we want that $\tilde{Q_i}$ be linearly equivalent to $d
{\pi_{\cf}}^* L -D(\K)$ ($L$ being a general line). The computation
of basis of the linear systems involves the resolution of systems of
linear equations. We are finding non-empty linear systems 
${\mathcal L}_d(\K)$ (whose projective dimension will be, a fortiori, equal to
0) whose unique curve $Q_i$ passes {\it effectively} through the
cluster $\K$.
}
\end{rem}
\subsection{Applications to the Harbourne-Hirschowitz
Conjecture}\label{hh}
Clusters of infinitely near points have been applied also to obtain
results dealing with the so-called Harbourne-Hirschowitz Conjecture
and related problems. Fixing $r+1$ points $P_0,P_1,\ldots,P_r$ of
$\gp^2$ in general position and given $r+1$ non-negative integers
$\underline{m}=(m_0,m_1,\ldots,m_r)$, the linear system ${\mathcal
L}_d(\underline{m})$ of plane projective curves of fixed degree $d$
having multiplicity $m_i$ (or larger) at $P_i$ for each $i$, has an
expected dimension (attained when all the conditions being imposed
are independent):
$$\edim
\cL_d(\underline{m}):=\max \left\{\frac{d(d+3)}{2}-\sum_{i=0}^n
\frac{m_i(m_i+1)}{2},-1\right\}.$$ Those systems whose dimension is
larger than the expected one are called {\it special}. The
Harbourne-Hirschowitz Conjecture intends to give a description of
all special linear systems. One of the equivalent formulations of
this conjecture asserts that a linear system is special if and only
if it has a multiple fixed component such that its strict transform
on the surface obtained by blowing-up the points $P_0,P_1\ldots,P_r$
is a $(-1)$-curve. This conjecture goes back to B. Segre \cite{seg}
and it has been reformulated by several authors (see \cite{harb1},
\cite{gim1}, \cite{hir}, \cite{harb4}, \cite{cil}, \cite{cil2}, and
\cite{cil3} for a survey).
There exists an extensive literature either giving partial proofs of
the conjecture or dealing with related subjects. It is out of the
scope of this paper to give a global overview of the topic;
\cite{harb3}, \cite{Mir}, \cite{cil3} and references given therein
will be helpful for the interested reader. We mention here the
result given in \cite{roe3}, where the semicontinuity theorem and a
sequence of specializations to constellations of infinitely near
points are used to obtain an algorithm for computing an upper bound
for the least degree $d$ for which $r+1\geq 9$ general points of
given multiplicities $m_0,m_1,\ldots,m_r$ impose independent
conditions to the linear system of curves of degree $d$ (that is,
the {\it regularity} of the system of multiplicities); also, an
explicit formula for a bound is obtained when all the multiplicities
are equal to $m$: $d+2\geq (m+1)(\sqrt{r+2.9}+\pi/8)$ (the
Harbourne-Hirschowitz Conjecture implies that the imposed conditions
are independent when $d(d+3)\geq (r+1)m(m+1)-2$).
In \cite{mons} it is provided an unbounded family of systems of
multiplicities $(m_i)_{i=0}^r$ for which the Harbourne-Hirschowitz
Conjecture is satisfied (considering $\mathbb{C}$ as the base
field). This result is obtained specializing the $r+1$ general
points to the configuration of base points of the pencil at infinity
defined by an Abhyankar-Moh-Suzuki curve and using semicontinuity.
The statement is the following one:
\begin{thm}\label{amsc}
Let $\varrho$ be the pencil at infinity defined by an
Abhyankar-Moh-Suzuki curve that is not a line and let
$\C_{\I(\varrho)}=\{Q_0,Q_1,\ldots,Q_r\}$ be its constellation of
base points. Let $\underline{m}=(m_0,m_1,\ldots,m_r)$ be a system of
multiplicities such that the cluster
$(\C_{\I(\varrho)},\underline{m})$ satisfies the proximity
inequalities $m_i-\sum_{Q_j\rightarrow Q_i}m_j\geq 0$, $0\leq i\leq
r$, the second one being a strict inequality (that is,
$m_1-\sum_{Q_j\rightarrow Q_1}m_j>0$). If the linear system
(supported at general points) $\cL_d(\underline{m})$ is special,
then it has a multiple fixed component whose strict transform on the
surface obtained by blowing-up the general points is a $(-1)$-curve.
Furthermore, this curve is the line joining the points corresponding
with the multiplicities $m_0$ and $m_1$.
\end{thm}
Notice that the above result depends only on the P-Enriques diagram
associated with the resolution of the singularity at infinity of the
fixed Abhyankar-Moh-Suzuki curve, and not on the curve itself. These
P-Enriques diagrams are completely characterized (see \cite{fdb1}
and \cite{fdb2}) and each of them provides an unbounded family of
multiplicities for which the Harbourne-Hirschowitz Conjecture is
satisfied.
It is worth adding that some of the facts in which the proof of
Theorem \ref{amsc} is based are the above mentioned results
(Theorems \ref{thm32} and \ref{thm33}) on the structure of the
effective semigroup, the nef cone and the characteristic cone of the
surface obtained by eliminating the base points of a pencil defined
by a curve having one place at infinity. In addition, it is relevant
the fact that the configurations $\C_{\I(\varrho)}$ as in the
statement of Theorem \ref{amsc} are P-sufficient (Prop.
\ref{prop36}).
In \cite{mons} it is also generalized the algorithm given in
\cite{roe3} for bounding the regularity of a system of
multiplicities by using P-Enriques diagrams of pencils at infinity
associated with Abhyankar-Moh-Suzuki curves.

\end{document}